\documentclass[10pt]{article}          
\RequirePackage[T1]{fontenc}
\RequirePackage[left=4cm,right=4cm]{geometry}

\RequirePackage{graphicx}
\RequirePackage{mathptmx}      
\RequirePackage{amsmath,amssymb,mathtools,amsthm}
\RequirePackage{flushend}
\RequirePackage[numbers,sort&compress]{natbib}
\RequirePackage[colorlinks,citecolor=blue,urlcolor=blue,linkcolor=blue]{hyperref}
\RequirePackage{todonotes}
\RequirePackage{multicol}

\newtheorem{assumption}{Assumption}[section]{\bf}{\rm}
\newtheorem{definition}{Definition}[section]
\newtheorem{remark}{Remark}[section]
\newtheorem{lemma}{Lemma}[section]
\newtheorem{theorem}{Theorem}[section]
\newtheorem{proposition}{Proposition}[section]
\newtheorem{corollary}{Corollary}[section]

\RequirePackage[nameinlink,capitalize]{cleveref}
\RequirePackage{mathrsfs}
\crefname{assumption}{Asm.}{Assumptions}
\crefname{definition}{Def.}{Definitions}
\crefname{lemma}{Lem.}{Lemmata}
\crefname{remark}{Rem.}{Remarks}
\crefname{proposition}{Prop.}{Propositions}
\crefname{theorem}{Thm.}{Theorems}

\newcommand{\R}{\mathbb{R}}

\renewcommand{\epsilon}{\varepsilon}
\newcommand{\N}{\mathbb N}

\newcommand{\eps}{\epsilon}

\newcommand{\sL}{{\mathsf{L}\!}}
\newcommand{\sBV}{{\mathsf{BV}}}
\newcommand{\sW}{\mathsf W}
\newcommand{\sC}{\mathsf C}
\newcommand{\sTV}{\mathsf{TV}}

\newcommand{\mEF}{\mathcal{EF}}
  \newcommand{\sgn}{\ensuremath{\textnormal{sgn}}}
\newcommand{\e}{\mathrm e}

\newcommand{\equivd}{\ensuremath{\vcentcolon\equiv}}
\newcommand{\OT}{{\Omega_{T}}}
\newcommand{\loc}{\textnormal{loc}}

\newcommand{\dd}{\ensuremath{\,\mathrm{d}}}
\DeclareMathOperator{\supp}{supp}
\DeclareMathOperator*{\esssup}{ess-\sup}
\DeclareMathOperator*{\essinf}{ess-\inf}
\newcommand{\cW}{\ensuremath{\mathcal{W}}}

\newcommand{\uv}{\underline{v}}

\allowdisplaybreaks[4]
\begin{document}
\title{On the singular limit problem for a discontinuous nonlocal conservation law}
\author{Alexander Keimer\thanks{Friedrich-Alexander-University Erlangen-Nürnberg (FAU), Department of Mathematics, Cauerstraße 11, 91058 Erlangen, Germany}, Lukas Pflug\thanks{Friedrich-Alexander-University Erlangen-Nürnberg (FAU), Department of Mathematics, Cauerstraße 11\ \& Competence Unit for Scientific Computing (CSC), Martensstraße 5a, 91058 Erlangen,} 
}
\date{\today}

\maketitle


\begin{abstract}
In this contribution we study the singular limit problem of a nonlocal conservation law with a discontinuity in space. The specific choice of the nonlocal kernel involving the spatial discontinuity as well enables it to obtain a maximum principle for the nonlocal equation. The corresponding local equation can be transformed diffeomorphically to a classical scalar conservation law where the well-know Kru\v{z}kov theory can be applied. However, the nonlocal equation does not scale that way which is why the study of convergence is interesting to pursue. For exponential kernels of the nonlocal operator, we establish the converge to the corresponding local equation under mild conditions on the involved discontinuous velocity. We illustrate our results with some numerical examples.
\end{abstract}
\begin{description}
\item[Keywords:]Nonlocal Conservation Law, Discontinuous nonlocal conservation law, Existence, Uniqueness, Singular limit, Entropy solution, Convergence nonlocal to local
\item[Mathematics Subject Classification:] 35L65, 35L99, 34A36
\end{description}

\maketitle

\section{Introduction}
Nonlocal conservation laws have received increasing attention over last decade. This not only due to in many cases more realistic description of dynamical phenomena\cite{armbruster,coron,scialanga,blandin2016well,keimer2,colombo_nonlocal,amadori,aggarwal,piccoli2018sparse,pflug2020emom,colombo2015nonlocal,rossi2020well,bayen2020modeling}, but also for their mathematical properties of not requiring an Entropy condition for uniqueness, etc. \cite{pflug,spinola,pflug3,crippa2013existence,colombo,kloeden,lorenz2020viability,lorenz2019nonlocal}. One problem which has been open for this decade consists of the singular limit problem which can roughly been translated into whether the unique weak solution of the nonlocal conservation law converges to the weak entropy solution of the corresponding local conservation law when the nonlocal weight converges in some sense to a Dirac distribution \cite{teixeira,spinolo}.

And indeed, for specific kernels and sign restricted initial datum this can be answered positively \cite{pflug4,coclite2022general,COLOMBO20211653,bressan2021entropy,bressan2019traffic,keimer42,colombo2022nonlocal}.
A related question consists of whether one can obtain a similar result when the nonlocal conservation laws has a discontinuity in space and this is what we will tackle in this contribution. 
Discontinuous conservation laws and existence and uniqueness of solutions have only been studied recently in \cite{chiarello2021existence,chiarello2021nonlocal,keimer2021discontinuous}.

For a discontinuous nonlocal conservation law as introduced in \cref{defi:discontinuous_nonlocal_conservation_law} and where the spatial discontinuous function does not only act multiplicatively on the velocity but also appears in the nonlocal term as a weight, we establish the existence and uniqueness of solutions as well as a maximum principle and stability results when smoothing initial datum and discontinuous part.

We also identify the corresponding local discontinuous conservation law in \cref{defi:local_conservation_laws}. However, this local conservation law can be dealt with by transforming the spatial variable diffeomorphically to end up with a classical local conservation law without discontinuity which inspires the definition of the weak solution for the discontinuous local conservation law in \cref{defi:local_conservation_laws} and which gives existence and uniqueness of solutions. Under specific assumptions on the discontinuous part of the velocity that it is OSL (one-sided Lipschitz-continuous) from above, we show the convergence of the nonlocal discontinuous conservation law to the local solution when the nonlocal weight converges to a Dirac distribution. 
It thus is the first result on any kind of singular limit convergence of discontinuous nonlocal conservation laws to their local counterpart while -- as detailed later in \cref{subsec:literature} -- there exists a significant interested for discontinuous local conservation laws.

In equations, we assume for the remainder of this contribution the following:
\begin{assumption}[Input datum]\label{ass:input_datum}
We assume that\\[-20pt]
\begin{multicols}{2}
\begin{itemize}
\item \(v\in\sL^{\infty}(\R;\R_{\geq\uv}) \cap \sTV(\R)\) for  \(\uv\in\R_{>0}\)
\item \(q_{0}\in\sL^{\infty}(\R;\R_{\geq0})\cap \sL^{1}(\R)\)
\item \(V\in\sW^{2,\infty}_{\loc}(\R):\) \(V'\leqq 0\)
\item \(\eta\in\R_{>0}\)
\end{itemize}
\end{multicols}
\end{assumption}
and study a nonlocal approximation of the following set of discontinuous (local) conservation laws (where the nonlocal version of that is found in \cref{defi:discontinuous_nonlocal_conservation_law}).
\begin{definition}[Discontinuous (local) conservation laws]\label{defi:discontinous_local}
For \(v\in\sL^{\infty}(\R;\R_{\geq\uv})\cap \sTV(\R),\ \uv\in\R_{>0}\),
the local conservation law with discontinuity in space considered in this work consists of the following Cauchy problem in the density \(q:\OT\times\R\)
\begin{align}
q_{t}(t,x)+\partial_{x}\big(V\big(v(x)q(t,x)\big)v(x)q(t,x)\big)&=0,&& (t,x)\in\OT\label{eq:q}\\
q(0,x)&=q_{0}(x), && x\in\R
\intertext{which is formally equivalent (setting \(\rho\equiv v\cdot q\) and multiplying with \(v\)) to }
\rho_{t}(t,x)+ v(x) \partial_{x}f(\rho(t,x))&=0,&& (t,x)\in\OT\label{eq:local_2}\\
\rho(0,x)&=v(x)q_{0}(x),&& x\in\R
\intertext{and can be transformed via the spatial transform \(\rho(t,x)=p\big(t,F(x)\big),\ (t,x)\in\OT\) into}
p_{t}(t,x)+\partial_{x}f(p(t,x))&=0,&& (t,x)\in\OT\label{eq:p}\\
p(0,x)&=v\big(F^{-1}(x)\big)q_{0}\big(F^{-1}(x)\big), && x\in\R
\end{align}
with \(F(x)\coloneqq \int_{0}^{x}\tfrac{1}{v(y)}\dd y\) and \(f(x)= xV(x)\ \forall x\in\R\) for a given \(V\in W^{1,\infty}_{\loc}(\R)\).
\end{definition}
The corresponding nonlocal discontinuous equation then reads as
\begin{definition}[The nonlocal discontinuous conservation law]\label{defi:discontinuous_nonlocal_conservation_law}
The nonlocal conservation law in \(q:\OT\rightarrow\R\) has the following form 
\begin{align*}
    q_{t}(t,x)+\partial_{x}\Big(V\big(\cW[v\cdot q](t,x)\big)v(x)q(t,x) \Big)&=0 , &&(t,x)\in\OT\\
    q(0,x)&=q_{0}(x), &&x\in\R
\end{align*}
with datum as in \cref{ass:input_datum} and the nonlocal operator \(\cW\) being defined as
\begin{align}
    \cW[v\cdot q](t,x)\coloneqq\tfrac{1}{\eta}\int_{x}^{\infty}\exp\big(\tfrac{x-y}{\eta}\big)v(y)q(t,y)\dd y,&& (t,x)\in\OT.\label{defi:nonlocal_term}
\end{align}
We call  \(q_{\eta}:\OT\rightarrow\R\) the solution of the nonlocal conservation law. 
\end{definition}
\begin{remark}[Scaling of the nonlocal conservation law as suggested in \cref{defi:discontinous_local}]\label{rem:scaling_nonlocal}
By setting \(\rho_{\eta}\coloneqq v(x)q_{\eta}(t,x)\) we obtain analogously to \cref{eq:local_2} its nonlocal approximation
\begin{align}
    \partial_{t}\rho_{\eta}(t,x)+v(x)\partial_{x}\Big(V\big(\cW[\rho_{\eta}](t,x)\big)\rho_{\eta}(t,x) \Big)&=0,\ &&(t,x)\in\OT\label{eq:rho_eta}
\end{align}
but by setting \(\rho_{\eta}(t,x)=p_{\eta}\Big(t,\int_{0}^{x}\tfrac{1}{v(y)}\dd y\Big),\ (t,x)\in\OT\) we do not get the analogue to \cref{eq:p} as the nonlocal term is not scaling invariant under diffeomorphisms.
\end{remark}
\begin{remark}[The nonlocal kernel in \cref{defi:nonlocal_term}]\label{rem:nonlocal_kernel}
Note that in \cref{defi:nonlocal_term} we have used as nonlocal kernel the exponential kernel. Although some of the later presented results will hold for more general kernels which are monotonically decreasing and of \(\sTV\) regularity (compare in particular \cite{keimer2021discontinuous}, the singular limit analysis strongly relies similarly to \cite{coclite2022general} on the exponential term. This is the reason why we consider from the beginning only this kernel. However, following some ideas in \cite{colombo2022nonlocal} and \cite{keimer42} it might be possible to generalize to a broader set of kernels.
\end{remark}

\subsection{Further related literature}\label{subsec:literature}
The literature on discontinuous local conservation laws is vast. For a short overview, see \cite{buerger2008conservation} where the density dependent flux function changes its form discontinuously dependent on the spatial location. To show, however, that the study of discontinuities in local conservation laws has drawn significant attention over last decades, we shortly revisit some of the considered problems. \cite{gimse1993conservation}, the flux itself is (at finitely many points) discontinuous with respect to the density, and Riemann problem as well as existence and uniqueness of solutions is studied. In \cite{towers2000convergence} a finite difference scheme for conservation laws is introduced where the discontinuity is a space-dependent function and enters the flux in a multiplicative way. \cite{diehl_1996_conservation,diehl_1995_scalar} considers a conservation law with discontinuous (in the spatial variable) flux function and a Dirac measure as right hand side modelling sedimentation, and establishes the existence and uniqueness of solutions by a generalized entropy condition. For modelling oil reservoirs  \cite{gimse_1992_solution} considers another discontinuous (in the spatial variable) conservation law with wave front tracking.
Eventually, \cite{adimurthi2004godunov,adimurthi2007explicit,mishra2007existence} consider a variety of different discontinuous conservation laws with discontinuities in the density dependent flux and the spatial coordinate.

\section{The (local) discontinuous conservation laws considered}
Now, we focus first on deriving a suitable definition for solutions of the discontinuous (local) conservation law in \cref{eq:q}.
Having identified the local conservation law where we expect the nonlocal solution to converge to, i.e., \cref{eq:q}, we state it as follows:
\begin{definition}[The local Cauchy problem]\label{defi:local_conservation_laws}
Let \cref{ass:input_datum} be given and define the flux as \(f(x)\coloneqq xV(x)\ \forall x\in\R,\) we consider
\begin{align*}
    q_{t}(t,x)+ \partial_{x}f\big(v(x)q(t,x)\big)&=0, && (t,x)\in\OT\\
    q(0,x)&=q_{0}(x),    && x\in\R
\end{align*}    
and call \(q\) (the existence is handled in \cref{lem:existence_entropy} later) the solution to the \textbf{local discontinuous conservation law}.
\end{definition}
First, we will need to define what we mean with solutions in this ``non-standard'' form. Notice that one could indeed aim for reformulating this problem in non-conservative form, however this requires additional regularity on \(v\).
This is, why we define the solution as follows using the gained insights in \cref{defi:discontinous_local} which enable the reformulation of the problem into a classical (local) conservation law.
\begin{definition}[Entropy solution]\label{defi:entropy_solution}
We call \(q\in \sC\big([0,T];\sL^{1}(\R)\big)\cap\sL^{\infty}((0,T);\sL^{\infty}(\R))\) a \textbf{weak Entropy solution} of the discontinuous conservation law in \cref{defi:local_conservation_laws} iff it satisfies the following Entropy inequality
for all \(\alpha\in\sW^{2,\infty}_{\loc}(\R)\) being convex, \(\beta\in \sW^{1,\infty}_{\loc}(\R)\) with \(\beta'\equiv\alpha'\cdot f'\) and for all \(\phi\in \sW^{1,\infty}_{\text{c}}\big((-42,T)\times\R;\R_{\geq0}\big)\)
\begin{equation}
\begin{gathered}
\mathcal{E}[\phi,\alpha,q]\coloneqq\iint_{\OT}\alpha(q(t,x)v(x))\tfrac{1}{v(x)}\phi_{t}(t,x)+\beta(q(t,x)v(x))\phi_{x}(t,x)\dd x\dd t\\
+\int_{\R}\alpha(q_{0}(x)v(x))\phi(0,x)\tfrac{1}{v(x)}\dd x\geq 0.
\end{gathered}
\label{eq:Entropy}
\end{equation}
\end{definition}
\begin{remark}[Lipschitz-continuous test functions]
In \cref{defi:entropy_solution} we have chosen test functions with compact support which are Lipschitz-continuous instead of smooth test functions. However, this is not a restriction as these functions can be approximated by smooth functions in  \(L^{1}_{\loc}\). The reason why we chose the larger set of test functions is due to the steps required in \cref{lem:existence_entropy} where we can only make sure that the test functions belong to the suggested Lipschitz class.
\end{remark}
Next, we prove that under the given and adjusted entropy condition in \cref{eq:Entropy} in \cref{defi:entropy_solution}, there exists a unique solution to the local conservation law in \cref{defi:local_conservation_laws}.
\begin{lemma}[Existence and uniqueness of Entropy solutions]\label{lem:existence_entropy}
The Cauchy problem in \cref{defi:local_conservation_laws} has a unique Entropy solution in the sense of \cref{defi:entropy_solution}.
\end{lemma}
\begin{proof}
We perform a substitution in the entropic formulation of solutions in \cref{eq:Entropy} and have -- recalling that instead of \(\phi\) we can also plug in as test-function \((t,x)\mapsto\phi(t,F(x))\) thanks to the regularity of \(\phi\) and \(F(x)\coloneqq\int_{0}^{x}\tfrac{1}{v(y)}\dd y\ \forall x\in\R\) -- \(v\) is positive and bounded away from zero so that \(F\) is a strictly monotone function and invertible as well as weakly differentiable with the inverse function also differentiable)
\begin{align*}
    0&\leq E\Big[\phi\big(\cdot,F(\ast)\big),\alpha,q\Big]\\
    &=\iint_{\OT}\alpha(q(t,x)v(x))F'(x)\phi_{t}(t,F(x))+\beta(q(t,x)v(x))\partial_{x}\phi(t,F(x))\dd x\dd t\\
    &\qquad +\int_{\R}\alpha(q_{0}(x)v(x))\phi(0,F(x))F'(x)\dd x\\
    &=\iint_{\OT}\alpha(q(t,x)v(x)){F}'(x)\phi_{t}(t,F(x))+\beta(q(t,x)v(x))\partial_{2}\phi(t,F(x))F'(x)\dd x\dd t\\
    &\qquad +\int_{\R}\alpha(q_{0}(x)v(x))\phi(0,F(x))F'(x)\dd x
    \intertext{and performing for \(x\in\R\) the substitution \(y=F(x)\)}
    &=\iint_{\OT}\alpha\Big(q\big(t,F^{-1}(y)\big)v\big(F^{-1}(y)\big)\Big)\phi_{t}(t,y) + \beta\Big(q\big(t,F^{-1}(y)\big)v\big(F^{-1}(y)\big)\Big)\partial_{y}\phi(t,y)\dd y\\
    &\qquad +\int_{\R}\alpha\Big(q_{0}\big(F^{-1}(y)\big)v\big(F^{-1}(y)\big)\Big)\phi(0,y)\dd y
\end{align*}
we observe that the last expression states the definition of Entropy solutions for the (spatially independent) conservation law in \(p\equiv v\big(F^{-1}(\cdot)\big) q\big(\ast,F^{-1}(\cdot)\big)\)
\begin{align*}
\partial_{t}p(t,x)+\partial_{x}f(p(t,x))&=0, &&(t,x)\in\OT\\
p(0,x)&=v\big(F^{-1}(x)\big)q_{0}\big(F^{-1}(x)\big), && x\in\R.
\end{align*}
However, for the equation in \(p\) we know that it admits a unique solution by the classical theorems \cite[Theorem 6.3]{bressan}, \cite[Theorem 19.1]{eymard},
 \cite[Theorem 2, Theorem 5, Section 5 Item 4]{kruzkov}, \cite{godlewski} and as \(F\in\sW^{1,\infty}_{\loc}(\R)\) with \(F'\in\sL^{\infty}(\R;\R_{\geq v})\) and \(v\in\sL^{\infty}(\R;\R_{\geq \uv})\) this uniqueness carries over to \(q\) as well which concludes the proof.
\end{proof}
Roughly speaking, thanks to the spatial transform \(F\) we were capable of reinterpreting the discontinuous Cauchy problem in \cref{defi:local_conservation_laws} as a spatially independent classical conservation law which we illustrated formally already in \cref{defi:discontinous_local} and \cref{eq:p}.
\begin{lemma}[Strictly concave/convex flux and Entropy condition]\label{lem:strict_concave_flux}
Assume that \(f\) as in \cref{defi:local_conservation_laws},  i.e.,\ \(f(x)=xV(x)\ \forall x\in\R\) is strictly concave, i.e.,\ 
\[
xV''(x)+2V'(x)<0,\  x\in \big(\essinf_{x\in\R} q_{0}(x),\|q_{0}\|_{\sL^{\infty}(\R)}\big) \text{ a.e.}.
\]
Then, for the uniqueness result in \cref{lem:existence_entropy} to hold it is enough that the Entropy inequality in \cref{defi:entropy_solution}, namely \cref{eq:Entropy}, holds for \textbf{one} strictly convex Entropy flux pair, for instance for 
\[
\alpha(x)=x^{2},\quad \forall x\in\R.
\]
\end{lemma}
\begin{proof}
The proof is a minor modification (with the outlined transforms in the proof of \cref{lem:existence_entropy}) of the results in \cite{westdickenberg,panov1994uniqueness}.
\end{proof}
\section{Nonlocal discontinuous conservation laws}
In this section, we introduce the nonlocal conservation law with spatial discontinuity which we will later prove to converge to the entropy solution in \cref{defi:entropy_solution} of the local conservation law as stated in \cref{defi:local_conservation_laws}.
In \cite{keimer2021discontinuous} we have shown the well-posedness of a discontinuous (in space) nonlocal conservation law of the form
\begin{equation}
\begin{aligned}
    q_{t}(t,x)+\partial_{x}\Big(v(x)V\big(\cW_{\eta}[q](t,x)\big)q(t,x) \Big)&=0,&& (t,x)\in\OT\\
    q(0,x)&=q_{0}(x),&& x\in\R\\
\cW_{\eta}[q](t,x)&\coloneqq\tfrac{1}{\eta}\int_{x}^{\infty}\exp\big(\tfrac{x-y}{\eta}\big)q(t,y)\dd y&& (t,x)\in\OT
\end{aligned}
\label{eq:nonlocal_conservation_law}
\end{equation}
with \(v\in\sL^{\infty}(\R;\R_{\geq\uv})\cap\sTV(\R),\ \uv\in\R_{>0}\) being the \textbf{discontinuous} part of the velocity, \(V\in\sW^{1,\infty}_{\loc}(\R)\) the \textbf{Lipschitz-continuous} part of the velocity and the nonlocal operator \(\cW_{\eta}\) with looking ahead parameter \(\eta\in\R_{>0}\). The exponential weight used here could -- as outlined before -- be replaced by another nonlocal weight being monotonically decreasing and total variation bounded, however, in this contribution, we stick to the exponential weight from the beginning and leave it to the reader to potentially generalize in the spirit of \cref{rem:nonlocal_kernel}.

The proposed dynamics has the disadvantage that classical maximum principles only hold in rather restricted setups (compare \cite[Theorem 3.3]{keimer2021discontinuous}). However, this changes significantly if we consider instead of the nonlocal operator \(\cW_{\eta}[q]\) the one which is enriched by the spatially discontinuous \(v\), namely \(\cW_{\eta}[v\cdot q]\).

But, given this different type of nonlocal conservation laws, the classical existence and uniqueness theorem as outlined in \cite[Theorem 3.1]{keimer2021discontinuous} cannot be applied directly. However, required adjustments are minor and we will thus state the existence and uniqueness result as well as an approximation result without proof.
For the sake of completeness, let us first define what we mean with solutions to the suggested nonlocal dynamics:
\begin{definition}[The discontinuous nonlocal conservation law and its weak solution]\label{defi:weak_solution_2}
Let \cref{ass:input_datum} be given, we call \(q\in \sC\big([0,T]; \sL^{1}(\R)\big)\cap\sL^{\infty}\big((0,T); \sL^{\infty}(\R)\big)\) a weak solution of \cref{defi:discontinuous_nonlocal_conservation_law}, namely
    \begin{equation}
\begin{aligned}
    q_{t}(t,x)&=-\partial_{x}\Big(v(x)V\big(\cW_{\eta}[v\cdot q](t,x)\big)q(t,x) \Big), && (t,x)\in\OT\\
    q(0,x)&=q_{0}(x),&&x\in\R\\
\cW_{\eta}[v\cdot q](t,x)&\coloneqq\tfrac{1}{\eta}\int_{x}^{\infty}\!\!\!\!\exp\big(\tfrac{x-y}{\eta}\big)v(y)q(t,y)\dd y,&& (t,x)\in\OT
\end{aligned}
\label{eq:nonlocal_term}
\end{equation}
iff the conditions in \cite[Def. 2]{keimer2021discontinuous} are met when using as the nonlocal operator
\[
\cW_{\eta}[v\cdot q]\ \text{ instead of }\ \cW_{\eta}[q].
\]
\end{definition}
Given the previous definition, we are in the position to state the existence and uniqueness result, supplemented by a maximum principle:
\begin{theorem}[Existence and uniqueness and a maximum principle]\label{theo:existence_uniqueness_maximum}
Given \cref{ass:input_datum}, the discontinuous nonlocal conservation law in \cref{defi:discontinuous_nonlocal_conservation_law} has on every finite time horizon \(T\in\R_{>0}\) a unique weak solution in the sense of \cref{defi:weak_solution_2}
\[
q\in \sC\big([0,T];\sL^{1}_{\loc}(\R)\big)\cap\sL^{\infty}\big((0,T);\sL^{\infty}(\R)\big)
\]
which satisfies the maximum principle in the variable ``\(v\cdot q\)''
\begin{align}
    \essinf_{y\in\R} v(y)\cdot q_{0}(y)\leq v(x)q(t,x)\leq \|v\cdot q_{0}\|_{\sL^{\infty}(\R)},\quad (t,x)\in\OT\text{ a.e.}.\label{eq:maximum_principle}
\end{align}
In particular, the solution \(q\) satisfies
\begin{align}
\tfrac{\essinf_{y\in\R} v(y)q_{0}(y)}{\|v\|_{\sL^{\infty}(\R)}}    \leq q(t,x)\leq \tfrac{\|v\cdot q_{0}\|_{\sL^{\infty}(\R)}}{\uv},\quad (t,x)\in\OT\text{ a.e.}.
\end{align}
\end{theorem}
\begin{proof}
The proof is a refinement of \cite[Thm, 3.1]{keimer2021discontinuous} and a weak stability result which we will detail in \cref{prop:weak_stability} and which enables it to consider sequences of solutions which are smooth and then passing to the limit. We do not go into details.
\end{proof}
As we require to smooth our solution to obtain \(\sTV\) bounds uniform in \(\eta\) we present the following stability result in \(\sC(\sL^{1})\):
\begin{proposition}[A stability result]\label{prop:weak_stability}
Let \(q_{0}\in \sL^{\infty}(\R)\cap \sL^{1}(\R;\R_{\geq0})\) be given and \[\big(q_{0}^{\eps}\big)_{\eps\in\R_{>0}}\subset \sC^{\infty}_{\text{c}}(\R):\ \lim_{\eps\rightarrow 0} \|q_{0}^{\eps}-q_{0}\|_{\sL^{1}(\R)}=0\] its mollified version with \(\|q_{0}^{\eps}\|_{\sL^{\infty}(\R)}\leq \|q_{0}\|_{\sL^{\infty}(\R)}\ \forall \eps\in\R_{>0}\).

Let the discontinuity \(v\in\sL^{\infty}(\R;\R_{\geq\uv})\cap \sTV(\R)\) be given and a smooth version
\[
v^{\eps}\in\sC\big(\R;\R_{\geq\uv}\big): \lim_{\eps\rightarrow 0}\|v-v^{\eps}\|_{\sL^{1}(\R)}=0,
\]
again with \(\|v_{\eps}\|_{\sL^{\infty}(\R)}\leq \|v\|_{\sL^{\infty}(\R)}\ \forall \eps\in\R_{>0}\).
Then, the solution to the nonlocal conservation law \(q^{\eps}\in \sW^{1,\infty}(\OT)\) with initial datum \(q_{0}^{\eps}\) and discontinuity \(v^{\eps}\) satisfies
\[
\lim_{\eps\rightarrow 0}\|q^{\eps}-q\|_{\sC([0,T];\sL^{1}(\R)}=0,
\]
if \(q\in \sC([0,T];\sL^{1}_{\loc}(\R))\cap \sL^{\infty}((0,T);\sL^{\infty}(\R))\) denotes the solution to the nonlocal conservation law with initial datum \(q_{0}\) and discontinuity \(v\).

In particular, the solution \(q^{\eps}\) is a strong solution.
\end{proposition}
\begin{proof}
The proof is a refinement of \cite[Thm. 3.2 \& Lem. 3.1]{keimer2021discontinuous}. 
That the proposed mollifiers exist is a direct consequence of \cite[Rem. C.18, ii]{leoni}.
\end{proof}
As we later require a stability result for the corresponding nonlocal operator \(\cW\) as well, we detail it here and mentioning that the convergence is then uniformly. This is due to the fact that the integral operator converts the \(\sL^{1}\) convergence to a uniform convergence as we look at a type of anti-derivative of the solution:
\begin{corollary}[Convergence of the nonlocal operator \(\cW_{\eta}\)]
Let the assumptions in \cref{prop:weak_stability} hold, the convergence of the solution carries over to the nonlocal term, i.e., it holds that
\[
\lim_{\eps\rightarrow 0}\big\|\cW_{\eta}[v^{\eps}\cdot q^{\eps}]-\cW_{\eta}[v\cdot q]\big\|_{\sL^{\infty}((0,T);\sL^{\infty}(\R))}=0.
\]
\end{corollary}
\begin{proof}
This follows directly from the definition of the nonlocal operator. To this end, let \((t,x)\in\OT\) be given, and compute for \(\eta\in\R_{>0}\)
\begin{align*}
&\Big|\cW_{\eta}[v^{\eps}\cdot q^{\eps}](t,x)-\cW_{\eta}[v\cdot q](t,x)\Big|\\
&\leq \Big|\cW_{\eta}[v^{\eps}\cdot q^{\eps}-v\cdot q](t,x)\Big|\\
&\leq \Big|\cW_{\eta}[v^{\eps}\cdot q^{\eps}-v\cdot q^{\eps}](t,x)\Big|+\Big|\cW_{\eta}[v\cdot q^{\eps}-v\cdot q](t,x)\Big|\\
&\leq \tfrac{1}{\eta}\|q_{0}\|_{\sL^{\infty}(\R)}\|v^{\eps}-v\|_{\sL^{1}(\R)}+ \tfrac{1}{\eta}\|v\|_{\sL^{\infty}(\R)}\big\|q^{\eps}-q\big\|_{\sC([0,T];\sL^{1}(\R))}
\end{align*}
where the previous estimates follow from the specific choices of the mollifiers, the triangular inequality and $1$ as  the uniform bound of the exponential function on \(\R_{>0}\).
From the last term, the conclusion follows by the strong convergence of \(v^{\eps},q^{\eps}_{\eta}\) as guaranteed by \cref{prop:weak_stability}.
\end{proof}
\section{\texorpdfstring{\(\sTV\)}{TV} bound for the nonlocal term}
In this section, we prove uniform \(\sTV\) bounds for the nonlocal term \(\cW_{\eta}\). Using the classical compactness results we then know that the the nonlocal term converges on a subsequence to a limit strongly in \(\sL^{1}_{\loc}\) and will enable it to also obtain the convergence of the solution \(q\).
\subsection{The dynamics in the nonlocal term}
To this end, we identify dynamics with regard to the nonlocal term \(\cW_{\eta}\) so that we can work fully on the nonlocal term and not on the equation in \(q\). The nonlocal term is smoother due to the involved integration and this is -- compare also \cite{coclite2022general} -- one of the reason why it makes sense to consider it instead of \(q\). As we require a derivative of \(v\) in the formula, we need to smooth the solution and consider smooth \(v\). However, as we have also \cref{prop:weak_stability}, we can directly assume that all involved functions are smooth and later pass to the limit.

\begin{lemma}[Transport equation for $\cW_{\eta}$ with nonlocal right hand side]\label{lem:nonlocal_equation_W}
Given the dynamics in \cref{eq:nonlocal_conservation_law} be given so that 
\[
q_{0}\in \sC^{\infty}_{\text{c}}(\R),\quad  v\in \sC^{\infty}(\R)\cap \sL^{\infty}(\R;\R_{\geq \uv})\cap \sTV(\R).
\]
Then, according to \cref{theo:existence_uniqueness_maximum} the nonlocal conservation law has a unique solution \(q_{\eta}\) and the nonlocal term \(\cW_{\eta}[v\cdot q]\eqqcolon\cW_{\eta}\) satisfies the following Cauchy problem for \(\eta\in\R_{>0}\)
    \begin{equation}
    \begin{aligned}
    \cW_{t}+v(x)V(\cW)\cW_{x}
    &=-\tfrac{1}{\eta}\!\!\int_{x}^{\infty}\!\!\!\!\e^{\frac{x-y}{\eta}}v(y)\cW(t,y)V'(\cW(t,y))\cW_{y}(t,y)\dd y\\
    &\qquad -\int_{x}^{\infty}\e^{\frac{x-y}{\eta}}v'(y)V(\cW(t,y))\cW_{y}(t,y)\dd y,&& (t,x)\in\OT,\\
    \cW(0,x)&=\tfrac{1}{\eta}\int_{x}^{\infty}\e^{\frac{x-y}{\eta}}v(y)q_{0}(y)\dd y, && x\in\R.
    \end{aligned}
    \label{eq:transport_equation_nonlocal}
    \end{equation}
\end{lemma}
\begin{proof}
Recalling the nonlocal term we have for \((t,x)\in\OT\)
\begin{align*}
 \cW_{\eta}(t,x)\coloneqq   \cW_{\eta}[v\cdot q](t,x)&=\tfrac{1}{\eta}\int_{x}^{\infty}\e^{\frac{x-y}{\eta}}v(y)q(t,y)\dd y.
\end{align*}
The derivative of \(\cW_{\eta}\) with regard to space can be computed as follows for \((t,x)\in\OT\)
\begin{align}
    \partial_{x}\cW_{\eta}(t,x)=\tfrac{1}{\eta}\big(\cW_{\eta}(t,x)-v(x)q(t,x)\big)\Rightarrow v(x)q(t,x)=\cW_{\eta}(t,x)-\eta\partial_{x}\cW_{\eta}(t,x).\label{eq:identity_W}
\end{align}
For the time derivative of \(\cW_{\eta}\) we have
\begin{align}
    \partial_{t}\cW_{\eta}(t,x)&=\partial_{t}\bigg(\tfrac{1}{\eta}\int_{x}^{\infty}v(y)\exp(\tfrac{x-y}{\eta})q(t,y)\dd y\bigg)=\tfrac{1}{\eta}\int_{x}^{\infty}v(y)\exp(\tfrac{x-y}{\eta})\partial_{t}q(t,y)\dd y\notag\\
\intertext{and as \(q\) is a strong solution of the discontinuous nonlocal conservation law \cref{defi:discontinuous_nonlocal_conservation_law}}
      &=-\tfrac{1}{\eta}\int_{x}^{\infty}v(y)\exp(\tfrac{x-y}{\eta})\partial_{y}\big(v(y)V(\cW_{\eta}(t,y))q(t,y)\big)\dd y\label{eq:42}
      \intertext{integration by parts}
      &=\tfrac{1}{\eta}v(x)^{2}V(\cW_{\eta}(t,x))q(t,x)+\tfrac{1}{\eta}\int_{x}^{\infty}v'(y)\exp(\tfrac{x-y}{\eta})v(y)V(\cW_{\eta}(t,y))q(t,y)\dd y\notag\\
      &\quad -\tfrac{1}{\eta^{2}}\int_{x}^{\infty}v(y)^{2}\e^{\frac{x-y}{\eta}}V(\cW_{\eta}(t,y))q(t,y)\dd y,\notag
      \intertext{taking advantage of identity \cref{eq:identity_W} to replace \(q\)}
      &=\tfrac{1}{\eta}v(x)V(\cW_{\eta}(t,x))\cW_{\eta}(t,x)-v(x)V(\cW_{\eta}(t,x))\partial_{x}\cW_{\eta}(t,x)\notag\\
      &\quad +\tfrac{1}{\eta}\int_{x}^{\infty}\!\!\!\!\!v'(y)\exp(\tfrac{x-y}{\eta})V(\cW_{\eta}(t,y))\cW_{\eta}(t,y)\dd y\notag\\
      &\quad-\int_{x}^{\infty}\!\!\!\!\!v'(y)\exp(\tfrac{x-y}{\eta})V(\cW_{\eta}(t,y))\partial_{y}\cW_{\eta}(t,y)\dd y\notag\\
      &\quad -\tfrac{1}{\eta^{2}}\int_{x}^{\infty}\!\!\!\!v(y)\e^{\frac{x-y}{\eta}}V(\cW_{\eta}(t,y))\cW_{\eta}(t,y)\dd y\notag\\
      &\quad+\tfrac{1}{\eta}\int_{x}^{\infty}\!\!\!\!v(y)\e^{\frac{x-y}{\eta}}V(\cW_{\eta}(t,y))\partial_{y}\cW_{\eta}(t,y)\dd y\notag
      \intertext{and an integration by parts in the latter term}
      &=-v(x)V(\cW_{\eta}(t,x))\partial_{x}\cW_{\eta}(t,x)-\int_{x}^{\infty}\!\!\!\!\!v'(y)\exp(\tfrac{x-y}{\eta})V(\cW_{\eta}(t,y))\partial_{y}\cW_{\eta}(t,y)\dd y\notag\\
      &\quad -\tfrac{1}{\eta}\int_{x}^{\infty}\!\!\!\!v(y)\e^{\frac{x-y}{\eta}}V'(\cW_{\eta}(t,y))\cW_{\eta}(t,y)\partial_{y}\cW_{\eta}(t,y)\dd y\notag
.\end{align}
However, this is indeed the equality which we wanted to establish. The corresponding initial datum is a direct consequence of the definition of the nonlocal \(\cW_{\eta}\).
\end{proof}

\subsection{A uniform \(\sTV\) bound}
Having the identity for the nonlocal term in \cref{lem:nonlocal_equation_W}, we can derive total variation estimates uniform in the nonlocal parameter \(\eta\in\R_{>0}\). This will be carried out in the following \cref{prop:TV_bounds}. However, before doing that we require a Lemma stating that the spatial derivative of the nonlocal operator vanishes at \(-\infty\).
\begin{lemma}[Vanishing \(\partial_{2}\cW_{\eta}\) at negative infinity]\label{lem:W_x_vanishing}
Let \cref{ass:input_datum} hold. Then, we have for \(q\in L^{\infty}(\R)\cap \sTV(\R)\cap \sC^{1}(\R)\) and \(v\in \sL^{\infty}(\R)\cap\sTV(\R)\cap\sC^{1}(\R)\) with \(\cW_{\eta}\) the nonlocal operator as in \cref{eq:nonlocal_term}
\[
\lim_{x\rightarrow-\infty} \partial_{x}\cW_{\eta}[v\cdot q](x)=0.
\]
\end{lemma}
\begin{proof}
We have for all \(x\in\R\) and \(y^{*}\in\R\)
\begin{align*}
    \big|\partial_{x}\cW_{\eta}[v\cdot q](x)\big|&= \big|\cW_{\eta}\big[v'\cdot q+v\cdot \partial_{x} q\big](x)\big|\\
    &\leq \tfrac{1}{\eta}\int_{-\infty}^{y^{*}} |v'(y)q(y)|\dd y+\tfrac{1}{\eta}\int_{-\infty}^{y^{*}} |v(y)\partial_{y}q(y)|\dd y\\
    &\quad +\tfrac{1}{\eta}\int_{y^{*}}^{\infty} \e^{\frac{x-y}{\eta}}|v'(y)q(y)|\dd y+\tfrac{1}{\eta}\int_{y^{*}}^{\infty} \e^{\frac{x-y}{\eta}}|v(y)\partial_{y}q(y)|\dd y\\
    &\leq \tfrac{1}{\eta}\|q\|_{\sL^{\infty}(\R)}|v|_{\sTV((-\infty,y^{*}))}+\tfrac{1}{\eta} \|v\|_{\sL^{\infty}(\R)} |q|_{\sTV((-\infty, y^{*}))}\\
    &\quad +\tfrac{1}{\eta}\int_{y^{*}}^{\infty} \e^{\frac{x-y}{\eta}}|v'(y)q(y)|\dd y+\tfrac{1}{\eta}\int_{y^{*}}^{\infty} \e^{\frac{x-y}{\eta}}|v(y)\partial_{y}q(y)|\dd y.
\end{align*}
Now, we chose \(y^{*}\in\R_{<0}\) negative enough so that \(|q|_{\sTV(-\infty,y^{*})}\) and \(|v|_{\sTV((-\infty,y^{*}))}\) are arbitrary small and letting \(x\rightarrow-\infty\) in the second two terms using the dominated convergence theorem, we have that both terms vanish. Altogether, we obtain
\[
\lim_{x\rightarrow -\infty}|\partial_{x}\cW_{\eta}[v\cdot q](x)|=0.
\]
\end{proof}
The next proposition will provide a uniform \(\sTV\) bound (with regard to \(\eta\)) of the solution which is ``partially'' uniform in \(\eps\) as well as long as \({v^{\eps}}'(x)\) is bounded from above for all \(\eps\in\R_{>0}\) and \(x\in\R\).
\begin{proposition}[\(TV\) bound for \(v\) with OSL condition (one sided Lipschitz)]\label{prop:TV_bounds}
Let \cref{ass:input_datum} hold and assume that we have initial datum and discontinuity smoothed with \(\eps\in\R_{>0}\) as in \cref{prop:weak_stability}. Then, for every \(\eps\in\R_{>0}\) the nonlocal equation in \(\cW^{\eps}_{\eta}\) as in \cref{lem:nonlocal_equation_W} satisfies
\begin{align}
    \big|\cW_{\eta}^{\eps}(t,\cdot)\big|_{\sTV(\R)}\leq |v^{\eps}\cdot q_{0}^{\eps}|_{\sTV(\R)}\e^{2t\esssup_{x\in\R}{v^{\eps}}'(x)\|V\|_{\sL^{\infty}((0,\|v\cdot q_{0}\|_{\sL^{\infty}(\R)}))}},\ \forall t\in [0,T].\label{eq:TV_estimate}
\end{align}
\end{proposition}
\begin{proof}
As \(\cW_{\eta}^{\eps}\) is according to \cref{prop:weak_stability} \(\sW^{2,\infty}(\OT)\) as \(q_{\eta}\in \sW^{1,\infty}(\OT)\) and \(\cW_{\eta}\big[v^{\eps}\cdot q_{\eta}^{\eps}\big]\) smooths the solution by one order), we differentiate through and obtain -- following the identity in \cref{lem:nonlocal_equation_W} -- and leaving out the dependency with regard to \(\eps\) and \(\eta\) and later the space time dependencies as well
\begin{align*}
\cW_{t,x}&=-v^{\eps}V(\cW)\cW_{xx}-v^{\eps}V'(\cW)\cW_{x}^{2}-{v^{\eps}}'V(\cW)\cW_{x} +\tfrac{1}{\eta}v^{\eps}\cW V'(\cW)\cW_{x}+{v^{\eps}}'V(\cW)\cW_{x}\\
&-\tfrac{1}{\eta^{2}}\int_{x}^{\infty}\exp(\tfrac{x-y}{\eta})v^{\eps}\cW V'(\cW)\cW_{y}\dd y-\tfrac{1}{\eta}\int_{x}^{\infty}\exp(\tfrac{x-y}{\eta}){v^{\eps}}'V(\cW)\cW_{y}\dd y\\
&=-v^{\eps}V(\cW)\cW_{xx}-v^{\eps}V'(\cW)\cW_{x}^{2}+\tfrac{1}{\eta}v^{\eps}\cW V'(\cW)\cW_{x}\\
&\quad -\tfrac{1}{\eta^{2}}\int_{x}^{\infty}\exp(\tfrac{x-y}{\eta})v^{\eps}\cW V'(\cW)\cW_{y}\dd y
-\tfrac{1}{\eta}\int_{x}^{\infty}\exp(\tfrac{x-y}{\eta}){v^{\eps}}'V(\cW)\cW_{y}\dd y.
\end{align*}
With this we can estimate the total variation as follows and have
\begin{align*}
    \tfrac{\dd}{\dd t}\int_{\R}|\cW_{x}(t,x)|\dd x&=\int_{\R} \sgn(\cW_{x}(t,x))\cW_{t,x}(t,x)\dd x\\
    &=-\int_{\R} \sgn(\cW_{x}(t,x))v^{\eps}(x)V(\cW(t,x))\cW_{xx}(t,x)\dd x\\
    &\quad -\int_{\R} \sgn(\cW_{x}(t,x))v^{\eps}(x)V'(\cW(t,x))\big(\cW_{x}(t,x)\big)^{2}\dd x\\
    &\quad-\int_{\R}\! \sgn(\cW_{x}(t,x))\tfrac{1}{\eta^{2}}\!\!\int_{x}^{\infty}\!\!\!\!\!\!\!\exp(\tfrac{x-y}{\eta})v^{\eps}(y)\cW(t,y) V'(\cW(t,y))\cW_{y}(t,y)\dd y\dd x\\
    &\quad -\int_{\R} \sgn(\cW_{x}(t,x))\tfrac{1}{\eta}\int_{x}^{\infty}\exp(\tfrac{x-y}{\eta}){v^{\eps}}'(y)V(\cW(t,y))\cW_{y}(t,y)\dd y\dd x\\
    &\quad +\tfrac{1}{\eta}\int_{\R} |\cW_{x}(t,x)|v^{\eps}(x)\cW(t,x) V'(\cW(t,x))\dd x
    \intertext{and an integration by parts in the first term}
    &=\int_{\R} |\cW_{x}(t,x)|{v^{\eps}}'(x)V(\cW(t,x))\dd x\\
    &\quad -\lim_{x\rightarrow\infty} |\cW_{x}(t,x)|v^{\eps}(x)V(\cW(t,x))+\lim_{x\rightarrow-\infty} |\cW_{x}(t,x)|v^{\eps}(x)V(\cW(t,x))\\
    &\quad+\tfrac{1}{\eta}\int_{\R} |\cW_{x}(t,x)|v^{\eps}(x)\cW(t,x) V'(\cW(t,x))\dd x\\
    &\quad-\int_{\R}\! \sgn(\cW_{x}(t,x))\tfrac{1}{\eta^{2}}\!\!\int_{x}^{\infty}\!\!\!\!\!\exp(\tfrac{x-y}{\eta})v^{\eps}(y)\cW(t,y) V'(\cW(t,y))\cW_{y}(t,y)\dd y\dd x\\
    &\quad -\int_{\R} \sgn(\cW_{x}(t,x))\tfrac{1}{\eta}\int_{x}^{\infty}\exp(\tfrac{x-y}{\eta}){v^{\eps}}'(y)V(\cW(t,y))\cW_{y}(t,y)\dd y\dd x\\
    \intertext{Decomposing $v^{\eps} \equiv v^{\eps,+} + v^{\eps,-}$ where \((v^{\eps,-})'\leqq 0\) and \((v^{\eps,+})' > 0\) and using \cref{lem:W_x_vanishing} for the boundary evaluation at \(x\rightarrow-\infty\) and leaving out the boundary evaluation at \(x\rightarrow\infty\) as it is non-positive}
    &\leq \int_{\R} |\cW_{x}(t,x)|\big((v^{\eps,+})'(x) + (v^{\eps,-})'(x)\big)V(\cW(t,x))\dd x\\
    &\quad+\tfrac{1}{\eta}\int_{\R} |\cW_{x}(t,x)|v^{\eps}\cW(t,x)V'(\cW(t,x))\dd x\\
    &\quad-\int_{\R}\! \sgn(\cW_{x}(t,x))\tfrac{1}{\eta^{2}}\!\int_{x}^{\infty}\!\!\!\!\exp(\tfrac{x-y}{\eta})v^{\eps}(y)\cW(t,y) V'(\cW(t,y))\cW_{y}(t,y)\dd y\dd x\\
    &\quad + \|(v^{\eps,+})'\|_{\sL^\infty(\R)}\tfrac{1}{\eta}\int_{\R} \int_{x}^{\infty}\exp(\tfrac{x-y}{\eta})V(\cW(t,y))|\cW_{y}(t,y)|\dd y\dd x\\ 
    &\quad -\int_{\R} \sgn(\cW_{x}(t,x))\tfrac{1}{\eta}\int_{x}^{\infty}\exp(\tfrac{x-y}{\eta})(v^{\eps,-})'(y)V(\cW(t,y))\cW_{y}(t,y)\dd y\dd x
    \intertext{and applying an change of order of integration and estimating the third and last term}
    &\leq \int_{\R} |\cW_{x}(t,x)|\big((v^{\eps,+})'(x) + (v^{\eps,-})'(x)\big)V(\cW(t,x))\dd x\\
    &\quad+\tfrac{1}{\eta}\int_{\R} |\cW_{x}(t,x)|v^{\eps}(x)\cW(t,x)V'(\cW(t,x))\dd x\\
    &\quad + \|(v^{\eps,+})'\|_{\sL^\infty(\R)}\tfrac{1}{\eta}\int_{\R} V(\cW(t,y))|\cW_{y}(t,y)| \int_{-\infty}^{y}\exp(\tfrac{x-y}{\eta})\dd x\dd y\\     
    &\quad -\tfrac{1}{\eta}\int_{\R}(v^{\eps,-})'(y)V(\cW(t,y))|\cW_{y}(t,y)| \int_{-\infty}^{y}\exp(\tfrac{x-y}{\eta})\dd x\dd y\\
    &\leq \int_{\R} |\cW_{x}(t,x)|\big((v^{\eps,+})'(x) + (v^{\eps,-})'(x)\big)V(\cW(t,x))\dd x\\
    &\quad+\tfrac{1}{\eta}\int_{\R} |\cW_{x}(t,x)|v^{\eps}(x)\cW(t,x)V'(\cW(t,x))\dd x\\
    &\quad + \|(v^{\eps,+})'\|_{\sL^\infty(\R)}\int_{\R} V(\cW(t,y))|\cW_{y}(t,y)| \dd y\\  
    &\quad -\int_{\R}(v^{\eps,-})'(y)V(\cW(t,y))|\cW_{y}(t,y)| \dd y\\
    &\leq 2\|(v^{\eps,+})'\|_{\sL^\infty(\R)} \|V\|_{\sL^\infty(0,\|\cW\|_{\sL^{\infty}(\R)})}|\cW(t,\cdot)|_{\sTV(\R)}.
\end{align*}
Applying Gr\"onwall's Lemma as in \cite[Appendix B k) ii]{evans} on the previous inequality, we obtain (reintroducing the dependencies)
\[
 \big|\cW_{\eta}^{\eps}(t,\cdot)\big|_{\sTV(\R)}\leq |\cW_{\eta}^{\eps}(0,\cdot)|_{\sTV(\R)}\exp\Big(2\big\|(v^{\eps,+})'\big\|_{\sL^\infty(\R)} \|V\|_{\sL^\infty(0,\|\cW_{\eta}^{\eps}\|_{\sL^{\infty}(\R)})}t\Big)\ \forall t\in[0,T].
\]
As 
\[
|\cW_{\eps}^{\eta}(0,\cdot)|_{\sTV(\R)}\leq \big|v^{\eps}\cdot q_{0}^{\eps}\big|_{\sTV(\R)},
\]
the claim follows when also recalling that according to \cref{theo:existence_uniqueness_maximum}
\[
\|\cW\|_{\sL^{\infty}((0,T);\sL^{\infty}(\R))}\overset{\eqref{eq:maximum_principle}}{\leq}\big\|v^{\eps}\cdot q_{0}^{\eps}\big\|_{\sL^{\infty}(\R)}.
\]
\end{proof}
\begin{remark}[Finiteness of \cref{eq:TV_estimate} in \cref{prop:TV_bounds} and more]
~
\begin{itemize}
\item
The finiteness of 
\[
\big|v^{\eps}\cdot q_{0}^{\eps}\big|_{\sTV(\R)}
\]
follows directly with the estimate
\[
\big|v^{\eps}\cdot q_{0}^{\eps}\big|_{\sTV(\R)}\leq \big\|v^{\eps}\big\|_{\sL^{\infty}(\R)}\big|q_{0}^{\eps}\big|_{\sTV(\R)}+\big\|q_{0}^{\eps}\big\|_{\sL^{\infty}(\R)}\big|v^{\eps}\big|_{\sTV(\R)}
\]
which is according to \cref{ass:input_datum} finite.
It is also worth mentioning that for discontinuities \(v^{\eps}\) which have a distributional derivative which is essentially bounded from above (it can not jump upwards) the total variation of the term grows maximally exponentially in time as stated in \cref{eq:TV_estimate}.
\item The upper OSL condition on the discontinuity of \(v\) as required in \cref{prop:TV_bounds} for obtaining \(\sTV\) bounds  on \(\cW_{\eta}^{\eps}\) (which later become uniform in \(\eta\) and \(\eps\) (see \cref{theo:strong_convergence_subsequences})) however lacks any interpretation and requires further analysis in the future to figure out whether this condition is purely technical and could be removed by an improved estimate or whether the lack of this condition would actually prevent the uniform \(\sTV\) bound to hold. On the numerical side, it seems like the nonlocal equation is converging to the corresponding local equation even for discontinuous \(v\) which can jump upwards (see \cref{sec:numerics}).
\end{itemize}
\end{remark}
\begin{lemma}[Total variation in time uniformly bounded]\label{lem:TV_bound_time}
Let the assumptions in \cref{prop:TV_bounds} hold. Then, the total variation in time of the nonlocal term \(\cW\) as in \cref{lem:nonlocal_equation_W} satisfies
for \(t\in[0,T]\)
\begin{equation}
\begin{aligned}
\|\partial_{t}\cW_{\eta}^{\eps}(t,\cdot)\|_{\sL^{1}(\R)}&\leq \|v^{\eps}\|_{\sL^{\infty}(\R)}\big(\|V\|_{\sL^{\infty}(X^{\eps})} +\|v^{\eps}\cdot q_{0}^{\eps}\|_{\sL^{\infty}(\R)}\|V'\|_{\sL^{\infty}(X^{\eps})}\big)\|\partial_{2}\cW_{\eta}^{\eps}(t,\cdot)\|_{\sL^{1}(\R)}\\
&\quad +2\big\|v^{\eps}\cdot q_{0}^{\eps}\big\|_{\sL^{\infty}(\R)}\|V\|_{\sL^{\infty}(X^{\eps})}\big\|{v^{\eps}}'\big\|_{\sL^{1}(\R)}
\end{aligned}
\label{eq:TV_bound_time}
\end{equation}
with \(X^{\eps}\coloneqq X[v^{\eps}\cdot g^{\eps}]\coloneqq \big(0,\|v^{\eps}\cdot q_{0}^{\eps}\|_{\sL^{\infty}(\R)}\big)\).
\end{lemma}
\begin{proof}
Thanks to the previous identity \cref{eq:transport_equation_nonlocal} in \cref{lem:nonlocal_equation_W} we can estimate the nonlocal term as follows and leave out again the dependency on \(\eta\in\R_{>0}\ni\eps\) to obtain for \(t\in[0,T]\)
\begin{align*}
    \int_{\R}|\cW_{t}(t,x)|\dd x&\leq\int_{\R}v(x)V(\cW(t,x))|\cW_{x}(t,x)|\dd x\\
    &\quad -\tfrac{1}{\eta}\int_{\R}\int_{x}^{\infty}\e^{\frac{x-y}{\eta}}v(y)\cW(t,y)V'(\cW(t,y))|\cW_{y}(t,y)|\dd y\dd x\\
    &\qquad +\int_{\R}\int_{x}^{\infty}\e^{\frac{x-y}{\eta}}|v'(y)|V(\cW(t,y))|\cW_{y}(t,y)|\dd y\dd x
    \intertext{estimating trivially and changing order of integration in the latter two terms}
    &\leq \|v\|_{\sL^{\infty}(\R)}\|V\|_{\sL^{\infty}((0,\|v\cdot q_{0}\|_{\sL^{\infty}(\R)}))}\int_{\R}|\cW_{x}(t,x)|\dd x\\
    &\quad -\tfrac{1}{\eta}\int_{\R}\int_{-\infty}^{y}\e^{\frac{x-y}{\eta}}\dd x\ v(y)\cW(t,y)V'(\cW(t,y))|\cW_{y}(t,y)|\dd y\\
    &\quad +\|\eta\partial_{2}\cW(t,\cdot)\|_{\sL^{\infty}(\R)}\tfrac{1}{\eta}\int_{\R}\int_{-\infty}^{y}\e^{\frac{x-y}{\eta}}\dd x\ |v'(y)|V(\cW(t,y))\dd x\\
    &\leq \|v\|_{\sL^{\infty}(\R)}\|V\|_{\sL^{\infty}((0,\|v\cdot q_{0}\|_{\sL^{\infty}(\R)}))}\int_{\R}|\cW_{x}(t,x)|\dd x\\
    &\quad -\int_{\R}\ v(y)\cW(t,y)V'(\cW(t,y))|\cW_{y}(t,y)|\dd y\\
    &\quad +\|\eta\partial_{2}\cW(t,\cdot)\|_{\sL^{\infty}(\R)}\int_{\R}\ |v'(y)|V(\cW(t,y))\dd x\\
    &\overset{\eqref{eq:identity_W}}{\leq}\|v\|_{\sL^{\infty}(\R)}\|V\|_{\sL^{\infty}((0,\|v\cdot q_{0}\|_{\sL^{\infty}(\R)}))}\int_{\R}|\cW_{x}(t,x)|\dd x\\
    &\quad +\|v\|_{\sL^{\infty}(\R)}\|\cW(t,\cdot)\|_{\sL^{\infty}(\R)}\|V'\|_{\sL^{\infty}((0,\|v\cdot q_{0}\|_{\sL^{\infty}(\R)}))}\int_{\R}\ |\cW_{y}(t,y)|\dd y\\
    &\quad +2\|v\cdot q(t,\cdot)\|_{\sL^{\infty}(\R)}\|V\|_{\sL^{\infty}((0,\|v\cdot q_{0}\|_{\sL^{\infty}(\R)}))}\int_{\R}\ |v'(y)|\dd x,
\end{align*}
from which the conclusion follows with the maximum principle in \cref{eq:maximum_principle} (\cref{theo:existence_uniqueness_maximum}) and the definition of the nonlocal operator \(\cW\) in \cref{eq:nonlocal_term}.
\end{proof}
In the following theorem, we use the previously obtained bounds on the total variation in time and space to pass to the limit for the non-smoothed nonlocal term.
\begin{theorem}[Strong convergence of subsequences of \(q_{\eta}\)]\label{theo:strong_convergence_subsequences}
Let \cref{ass:input_datum} hold. Assume in addition that  \(v\) is one-sided (upper) Lipschitz-continuous, i.e.,\ 
\begin{equation}
\exists L\in\R\ \forall (x,y)\in\R^{2},\ x\geq y:\ v(x)-v(y)\leq L(x-y)\ \label{eq:OSL}
.\end{equation}
Then, the nonlocal term in \cref{defi:weak_solution_2} \(\cW_{\eta}[v\cdot q]\) satisfies
\begin{align*}
    \sup_{\eta\in (0,1)}\big|\cW_{\eta}[v\cdot q](t,\cdot)\big|_{\sTV(\R)}\leq |v\cdot q_{0}|_{\sTV(\R)}\exp\Big(2tL\|V\|_{\sL^{\infty}(0,\|v\cdot q_{0}\|_{\sL^{\infty}(\R)})}\Big),\ \forall t\in[0,T].
\end{align*}
Moreover, there exists \(q_{*}\in \sC\big([0,T];\sL^{1}(\R)\big)\) and a subsequence \((\eta_{k})_{k\in\N}\subset \R: \lim_{k\rightarrow\infty} \eta_{k}=0\) so that the solution to the nonlocal conservation law in \cref{defi:weak_solution_2} \(q_{\eta}\in \sC([0,T];\sL^{1}(\R))\) as well as the nonlocal operator \(\cW_{\eta}[v\cdot q]\in \sW^{1,\infty}(\OT)\) satisfy
\[
\lim_{k\rightarrow\infty} \|q_{\eta_{k}}-q_{*}\|_{\sC([0,T];\sL^{1}_{\text{loc}}(\R))}=0\ \wedge \ \lim_{k\rightarrow\infty} \|\cW_{\eta_{k}}[v\cdot q_{\eta_{k}}]-v\cdot q_{*}\|_{\sC([0,T];\sL^{1}_{\text{loc}}(\R))}=0.
\]
\end{theorem}
\begin{proof}
Let \(\eta\in\R_{>0}\) be given, and smooth with a standard mollifier \(\{\phi_{\eps}\}_{\eps\in\R_{>0}}\subset \sC^{\infty}_{\text{c}}(\R)\) as in \cite[Rem. C.18, ii]{leoni} initial datum as well as discontinuity.
On the smoothed solution \(q^{\eps}_{\eta}\) we can apply \cref{prop:TV_bounds}
and have (now, we really make all dependencies visible, this is in particular the dependency of the solution \(q_{\eta}\) on the nonlocal term \(\eta\in\R_{>0}\))
\[
\big|\cW_{\eta}\big[v^{\eps}\cdot q^{\eps}_{\eta}\big](t,\cdot)\big|_{\sTV(\R)}\leq \big|v^{\eps}\cdot q_{0}^{\eps}\big|_{\sTV(\R)}\exp\Big(2t\esssup_{x\in\R}{v^{\eps}}'(x)\|V\|_{\sL^{\infty}(0,\|v\cdot q_{0}\|_{\sL^{\infty}(\R)})}\Big).
\]
As can be seen, the right hand side is uniform in \(\eta\) and as we have -- thanks to the standard mollifier -- 
\begin{gather*}
|v^{\eps}|_{\sTV(\R)}\leq |v|_{\sTV(\R)}\ \wedge\ \|v^{\eps}\|_{\sL^{\infty}(\R)}\leq \|v\|_{\sL^{\infty}(\R)}\ \wedge\ |q_{0}^{\eps}|_{\sTV(\R)}\leq |q_{0}|_{\sTV(\R)} \\ \wedge\ \|q_{0}^{\eps}\|_{\sL^{\infty}(\R)}\leq \|q_{0}\|_{\sL^{\infty}(\R)}\ \wedge\ \esssup_{x\in\R}{v^{\eps}}'(x)\leq L
\end{gather*}
we can let \(\eps\rightarrow 0\) and have for \(t\in[0,T]\)
\begin{equation}
\begin{split}
\sup_{\eta\in\R_{>0}}\big|\cW_{\eta}\big[v\cdot q_{\eta}\big](t,\cdot)\big|_{\sTV(\R)}&\leq \big(\|v\|_{\sL^{\infty}(\R)}|q_{0}|_{\sTV(\R)}+|v|_{\sTV(\R)}\|q_{0}\|_{\sL^{\infty}(\R)}\big)\\
&\qquad\cdot\exp\Big(2tL\|V\|_{\sL^{\infty}((0,\|v\cdot q\|_{\sL^{\infty}(\R)}))}\Big).
\end{split}
\label{eq:uniform_TV_estimate}
\end{equation}
This estimate can also be made uniform  in \(t\in[0,T]\) and we then have
\begin{equation}
\begin{split}
\sup_{t\in[0,T]}\sup_{\eta\in\R_{>0}}\big|\cW_{\eta}\big[v\cdot q_{\eta}\big](t,\cdot)\big|_{\sTV(\R)}&\leq \big(\|v\|_{\sL^{\infty}(\R)}|q_{0}|_{\sTV(\R)}+|v|_{\sTV(\R)}\|q_{0}\|_{\sL^{\infty}(\R)}\big)\\
&\qquad\cdot\exp\Big(2TL\|V\|_{\sL^{\infty}((0,\|v\cdot q\|_{\sL^{\infty}(\R)}))}\Big).
\end{split}
\label{eq:uniform_TV_estimate_time_uniform}
\end{equation}
As the right hand side is bounded, we can use a classical compactness argument to show the claim. We take advantage of \cite[Thm. 1, p.71]{simon} where the following is stated:
Let \(\big(B;\|\cdot\|_{B}\big)\) be a real Banach space and \(F\subset \sC([0,T];B)\). Then, \(F\) is relatively compact in \(\sC([0,T];B)\) if the following conditions are met:
\begin{enumerate}
    \item \(\forall (t_{1},t_{2})\in(0,T)^{2}: t_{1}< t_{2} \quad \Big\{\int_{t_{1}}^{t_{2}}f(t)\dd t:\ f\in F\Big\}\) is relatively compact in \(B\)\label{item:1}.
    \item \(\lim_{h\rightarrow 0}\sup_{f\in F}\|f(t+h)-f(t)\|_{\sC([0,T-h];B)}=0\).\label{item:2}
\end{enumerate}
In our setting, we have
\[
F\coloneqq \Big\{\cW_{\eta}[v\cdot q_{\eta}]\in\sC([0,T];\sL^{1}_{\loc}(\R)): \eta\in\R_{\geq 0},\ \cW_{\eta}[v\cdot q_{\eta}]\text{ as in \cref{defi:weak_solution_2}} \Big\},\ B\coloneqq \sL^{1}_{\loc}(\R).
\]
(Obviously, \(\sL^{1}_{\loc}(\R)\) is not a Banach space, so just think of that we do the compactness on any open bounded interval.)
We start with \cref{item:1} and have for \(\eta\in\R_{>0}\)
\begin{align*}
&\bigg|\int_{t_{1}}^{t_{2}}\cW_{\eta}[v\cdot q_{\eta}](s,\cdot)\dd s\bigg|_{\sTV(\R)}{\leq} \int_{t_{1}}^{t_{2}}\big|\cW_{\eta}[v\cdot q_{\eta}](s,\cdot)\big|_{\sTV(\R)}\dd s\\
&\overset{\eqref{eq:uniform_TV_estimate}}{\leq}\tfrac{\|v\|_{\sL^{\infty}(\R)}|q_{0}|_{\sTV(\R)}+|v|_{\sTV(\R)}\|q_{0}\|_{\sL^{\infty}(\R)}}{2L\|V\|_{\sL^{\infty}((0,\|v\cdot q\|_{\sL^{\infty}(\R)}))}}\Big(\e^{2t_{2}L\|V\|_{\sL^{\infty}((0,\|v\cdot q\|_{\sL^{\infty}(\R)}))}}-\e^{2t_{1}L\|V\|_{\sL^{\infty}((0,\|v\cdot q\|_{\sL^{\infty}(\R)}))}}\Big).
\end{align*}
However, this means that there is a uniform bound (with respect to \(\eta\)) on the total variation of 
\[
x\mapsto \int_{t_{1}}^{t_{2}}\cW_{\eta}[v\cdot q_{\eta}](s,x)\dd s
\]
and we can apply a type of Helly's compactness result \cite[Thm. 13.35]{leoni} to deduce that
\[
\bigg\{\int_{t_{1}}^{t_{2}}f(t)\dd t:\ f\in F\bigg\} \overset{c}{\hookrightarrow} \sL^{1}_{\loc}(\R)\ \forall (t_{1},t_{2})\in (0,T)^{2},\ t_{1}<t_{2}.
\]
As compactness implies relative compactness, we have met the requirements of \cref{item:1}.
For \cref{item:2}, let \(h\in(0,T),\) \(t\in (0,T-h)\) be given and estimate for \(\eta\in\R_{>0}\) and \(K\subset \R\) compact
\begin{align}
    &\int_{K}\big|\cW_{\eta}[v\cdot q_{\eta}](t+h,x)-\cW_{\eta}[v\cdot q_{\eta}](t,x)\big|\dd x\label{eq:time_estimate_compactness}\\
    &=\int_{K}\Big|\int_{t}^{t+h}\partial_{s}\cW_{\eta}[v\cdot q_{\eta}](s,x)\dd s\Big|\dd x\notag
    \intertext{introducing the smoothed version for \(\eps\in\R_{>0}\) as stated as well in \cref{prop:weak_stability,lem:TV_bound_time}}
    &\leq \int_{K}\Big|\int_{t}^{t+h}\!\!\!\!\!\!\!\partial_{s}\cW_{\eta}[v\cdot q_{\eta}](s,x)-\partial_{s}\cW_{\eta}[v^{\eps}\cdot q_{\eta}^{\eps}](s,x)\dd s\Big|\dd x+ \int_{K}\Big|\int_{t}^{t+h}\!\!\!\!\!\!\!\partial_{s}\cW_{\eta}[v^{\eps}\cdot q_{\eta}^{\eps}](s,x)\dd s\Big|\dd x\notag\\
    &\leq\int_{K}\Big|\int_{t}^{t+h}\!\!\!\!\!\!\!\partial_{s}\cW_{\eta}[v\cdot q_{\eta}](s,x)-\partial_{s}\cW_{\eta}[v^{\eps}\cdot q_{\eta}^{\eps}](s,x)\dd s\Big|\dd x
    + \!\int_{K}\int_{t}^{t+h}\!\big|\partial_{s}\cW_{\eta}[v^{\eps}\cdot q_{\eta}^{\eps}](s,x)\big|\dd s\dd x.\notag
\end{align}
The first term converges to zero for \(\eps\rightarrow 0,\) so we focus on the second and recall \cref{lem:TV_bound_time} to obtain
\begin{align*}
    &\int_{K}\!\!\int_{t}^{t+h}\!\!\!\big|\partial_{s}\cW_{\eta}[v^{\eps}\cdot q_{\eta}^{\eps}](s,x)\big|\dd s\dd x=\int_{t}^{t+h}\!\!\!\!\int_{K}\big|\partial_{s}\cW_{\eta}[v^{\eps}\cdot q_{\eta}^{\eps}](s,x)\big|\dd x\dd s\\
    &\leq h\!\sup_{t\in[0,T]}\int_{\R}\Big|\partial_{t}\cW_{\eta}[v^{\eps}\cdot q_{\eta}^{\eps}](t,x)\Big|\dd x\\
    &\quad\overset{\eqref{eq:TV_bound_time}}{\leq} h\sup_{t\in[0,T]} \bigg(\|v^{\eps}\|_{\sL^{\infty}(\R)}\big(\|V\|_{\sL^{\infty}((0,\|v^{\eps}\cdot q_{0}^{\eps}\|_{\sL^{\infty}(\R)}))} +\|v^{\eps}\cdot q_{0}^{\eps}\|_{\sL^{\infty}(\R)}\|V'\|_{\sL^{\infty}((0,\|v^{\eps}\cdot q_{0}^{\eps}\|_{\sL^{\infty}(\R)}))}\big)\\
&\qquad\qquad \cdot \|\partial_{2}\cW_{\eta}^{\eps}(t,\cdot)\|_{\sL^{1}(\R)}+2\big\|v^{\eps}\cdot q_{0}^{\eps}\big\|_{\sL^{\infty}(\R)}\|V\|_{\sL^{\infty}((0,\|v^{\eps}\cdot q_{0}^{\eps}\|_{\sL^{\infty}(\R)}))}\big\|{v^{\eps}}'\big\|_{\sL^{1}(\R)}\bigg)
\intertext{and thanks to the properties of the smoothing}
&\leq h\sup_{t\in[0,T]} \bigg(\|v\|_{\sL^{\infty}(\R)}\big(\|V\|_{\sL^{\infty}((0,\|v\cdot q_{0}\|_{\sL^{\infty}(\R)}))} +\|v\cdot q_{0}\|_{\sL^{\infty}(\R)}\|V'\|_{\sL^{\infty}((0,\|v\cdot q_{0}\|_{\sL^{\infty}(\R)}))}\big)\\
&\qquad\qquad \cdot \big|\cW_{\eta}(t,\cdot)\big|_{\sTV(\R)}+2\big\|v\cdot q_{0}\big\|_{\sL^{\infty}(\R)}\|V\|_{\sL^{\infty}((0,\|v\cdot q_{0}\|_{\sL^{\infty}(\R)}))}|v|_{\sTV(\R)}\bigg).
\end{align*}
The last term is uniform in \(\eps\in\R_{>0}\) and we have by \cref{eq:uniform_TV_estimate_time_uniform} that \(|\cW_{\eta}|_{\sTV(\R)}\) is uniformly bounded in \(\eta\in\R_{>0}\). The remaining terms are bounded as well so that we can take in \cref{eq:time_estimate_compactness} the supremum over \(t\in[0,T-h]\) and obtain \cref{item:2}.
Altogether, we have 
\[
F\overset{\text{r.c.}}{\hookrightarrow} \sC\big([0,T];\sL^{1}_{\loc}(\R)\big)
\]
and thus, we obtain a subsequence \(\{\eta_{k}\}_{k\in\N_{\geq 1}},\ \lim_{k\rightarrow\infty}\eta_{k}=0\) and \(\cW_{*}\in F\) so that
\begin{equation}
\lim_{k\rightarrow\infty} \|\cW_{\eta_{k}}-\cW_{*}\|_{\sC([0,T];\sL^{1}_{\loc}(\R))}=0.\label{eq:nonlocal_limit_point}
\end{equation}
Recalling \cref{eq:identity_W}, we have for every \(K\subset\R\) compact and \(t\in[0,T]\)
\[
\eta_{k}\big|\cW_{\eta_{k}}(t,\cdot)\big|_{\sTV(\R)}\geq \eta_{k}\big|\cW_{\eta_{k}}(t,\cdot)\big|_{\sTV(K)}=\|\cW_{\eta_{k}}(t,\cdot)-v(\cdot)q_{\eta_{k}}(t,\cdot)\|_{\sL^{1}(K)}
\]
and by \cref{eq:nonlocal_limit_point}, we thus have (recall that in particular \(v\in\sL^{\infty}(\R;\R_{\geq \uv})\)
\[
\lim_{k\rightarrow\infty} \|q_{\eta_{k}}-\tfrac{\cW_{*}}{v}\|_{\sC([0,T];\sL^{1}(K))},
\]
so that we can identify as limit point 
\begin{equation}
q_{*}\equiv \tfrac{\cW_{*}}{v}.\label{eq:limit_q_cW}
\end{equation}
This concludes the proof.
\end{proof}
\subsection{Convergence of the nonlocal weak solution to the local one -- Entropy admissibility}
Similar to what had been proven in \cite{bressan2021entropy} for nonlocal conservation laws without discontinuity, we show in this section that the strong convergence of the nonlocal solution \(q_{\eta}\rightarrow q^{*}\) implies that \(q^{*}\) itself satisfies the Entropy condition as outlined in \cref{defi:entropy_solution}. This is detailed in the following \cref{theo:convergence_entropy_solutions}:
\begin{theorem}[Strong convergence implies Entropy admissibility]\label{theo:convergence_entropy_solutions}
Let \cref{ass:input_datum} hold and assume that the solution to the discontinuous nonlocal conservation law in \cref{defi:discontinuous_nonlocal_conservation_law} converges strongly in \(\sC([0,T];\sL^{1}(\R))\) when \(\eta\rightarrow 0\) to a limit \(q_{*}\in\sC([0,T];\sL^{1}(\R))\), i.e.\ 
\[
\lim_{\eta\rightarrow 0} \|q_{\eta}-q_{*}\|_{\sC([0,T];\sL^{1}(\R))}=0.
\]
In addition, assume 
\begin{equation}
x\mapsto x V(x) \text{ strictly concave }\Longleftrightarrow 2V'(x)+x V''(x)<0,\ \forall x\in \big[0, \|v\cdot q_{0}\|_{\sL^{\infty}(\R)}\big].\label{eq:strict_concavity}
\end{equation}
Then, \(q_{*}\) is a \textbf{weak Entropy solution} in the sense of \cref{defi:entropy_solution} and thus, \(q_{*}\) is the (unique) Entropy solution of the discontinuous \textbf{local} conservation law in \cref{defi:discontinous_local}.
\end{theorem}
\begin{proof}
We plug the smoothed version of the nonlocal conservation law, i.e. \(\{q_{\eta}^{\eps}\}_{\eps\in\R_{>0}}\subset\sC^{1}(\OT)\) as in \cref{prop:weak_stability} into the Entropy condition \cref{eq:Entropy} and have with the abbreviation \(\rho^{\eps}_{\eta}\equivd v^{\eps}q_{\eta}^{\eps}\) and an integration by parts
\begin{align*}
 \mEF[\phi,\alpha,q_{\eta}^{\eps}]
 &=-\iint_{\OT}\phi(t,y)\Big(\alpha'(\rho^{\eps}_{\eta}(t,y))\partial_{t}\rho_{\eta}^{\eps}(t,y)\tfrac{1}{v^{\eps}(y)}+\beta'(\rho^{\eps}_{\eta}(t,y))\partial_{y}\rho^{\eps}_{\eta}(t,y)\Big)\dd y
 \intertext{\(\alpha(x)=x^{2}\ \forall x\in\R\) as we have strictly concave flux so that \cref{lem:strict_concave_flux} is applicable, and thus \(\beta'\equiv \alpha'\cdot f'\),\ \(f(x)=xV(x)\)}
 &=-\iint_{\OT}\phi(t,y)\Big(2\rho_{\eta}^{\eps}\partial_{t}\rho_{\eta}^{\eps}\tfrac{1}{v^{\eps}(y)}+2\rho_{\eta}^{\eps} f'(\rho_{\eta}^{\eps})\partial_{y}\rho_{\eta}^{\eps}\Big)\dd y\\
 \intertext{plugging in the strong form for \(\partial_{t}\rho^{\eps}_{\eta}\) as in \cref{eq:rho_eta} in \cref{rem:scaling_nonlocal}}
 &=-\iint_{\OT}\phi(t,y)\Big(-2\rho_{\eta}^{\eps}\partial_{y}\big(V\big(\cW_{\eta}[\rho_{\eta}^{\eps}]\big)\rho_{\eta}^{\eps}\big)+2\rho_{\eta}^{\eps} f'(\rho_{\eta}^{\eps})\partial_{y}\rho_{\eta}^{\eps}\Big)\dd y
 \intertext{and another integration by parts -- recalling that \(\phi\in \sW^{1,\infty}_{\text{c}}((-42,T)\times\R;\R_{\geq0})\) and is thus in particular compactly supported --}
 &=\underbrace{2\iint_{\OT}\!\!\!\!\partial_{y}\phi(t,y)\Big(-\rho_{\eta}^{\eps} V\big(\cW_{\eta}^{\eps}[\rho_{\eta}^{\eps}]\big)\rho_{\eta}^{\eps}+\rho_{\eta}^{\eps} V\big(\rho_{\eta}^{\eps}\big)\rho_{\eta}^{\eps}\Big)\dd y}_{\eqqcolon A}\\
 &\quad \underbrace{-2\iint_{\OT}\!\!\!\!\phi(t,y)\Big(\partial_{y}\rho_{\eta}^{\eps}V\big(\cW_{\eta}[\rho^{\eps}_{\eta}]\big)\rho_{\eta}^{\eps}-\partial_{y}\rho_{\eta}^{\eps}V(\rho_{\eta}^{\eps})\rho_{\eta}^{\eps}\Big)\dd y}_{\eqqcolon B}.
\end{align*}
 The first term \(A\) of the previous computation converges to zero as
 \begin{equation}
 \begin{aligned}
\tfrac{|A|}{2}&=\bigg|\iint_{\OT}\!\!\!\!\partial_{y}\phi(t,y)\Big(-\rho_{\eta}^{\eps} V\big(\cW_{\eta}[\rho_{\eta}^{\eps}]\big)\rho_{\eta}^{\eps}+\rho_{\eta}^{\eps} V\big(\rho_{\eta}^{\eps}\big)\rho_{\eta}^{\eps}\Big)\dd y\bigg|\\
&\leq \|\partial_{y}\phi\|_{\sL^{\infty}(\OT)}\big\|\rho^{\eps}_{\eta}\big\|_{\sL^{\infty}((0,T);\sL^{\infty}(\R))}^{2}\|V'\|_{\sL^{\infty}((0,\|q_{0}^{\eps}\cdot v^{\eps}\|_{\sL^{\infty}(\R)}))}\\
&\qquad\cdot\big\|\cW_{\eta}[\rho^{\eps}_{\eta}]-\rho^{\eps}_{\eta}\big\|_{\sC([0,T];\sL^{1}((-N,N))}
 \end{aligned}
 \label{eq:4212}
 \end{equation}
 with \(N\in\R_{>0}\) be chosen so that \(\supp(\phi)\subset (-N,N)\).

Recalling \cref{eq:identity_W}, we have
\[
\big\|\cW_{\eta}[\rho^{\eps}_{\eta}]-\rho^{\eps}_{\eta}\big\|_{\sC([0,T];\sL^{1}((-N,N))}\leq\eta\sup_{t\in[0,T]}\big|\cW_{\eta}[\rho^{\eps}_{\eta}](t,\cdot)\big|_{\sTV(\R)}
\]
and thanks to \cref{prop:TV_bounds}, \cref{theo:strong_convergence_subsequences} the total variation is bounded uniformly in \((\eta,\eps)\in\R^{2}\) so that for \(\eta\rightarrow0\) this term indeed vanishes.

We continue with the second term and compute furthermore
\begin{align*}
 B&=\iint_{\OT}\phi(t,y)\Big(2\partial_{y}\rho_{\eta}^{\eps}V(\rho_{\eta}^{\eps})\rho_{\eta}^{\eps}-2\partial_{y}\rho_{\eta}^{\eps}V\big(\cW_{\eta}[\rho_{\eta}^{\eps}]\big)\rho_{\eta}^{\eps}\Big)\dd y\\
 &=\iint_{\OT}\phi(t,y)2\partial_{y}\rho_{\eta}^{\eps}V(\rho_{\eta}^{\eps})\rho_{\eta}^{\eps}\dd y-\iint_{\OT}\phi(t,y)2\partial_{y}\rho_{\eta}^{\eps}V\big(W[\rho_{\eta}^{\eps}]\big)\rho_{\eta}^{\eps}\dd y\\
&=\iint_{\OT}\phi(t,y)2\partial_{y}\rho_{\eta}^{\eps}V(\rho_{\eta}^{\eps})\rho_{\eta}^{\eps}\dd y+\iint_{\OT}\phi_{y}(t,y)\big(\rho_{\eta}^{\eps}\big)^{2}V\big(\cW_{\eta}[\rho_{\eta}^{\eps}]\big)\dd y\\ &\quad +\iint_{\OT}\phi(t,y)\big(\rho_{\eta}^{\eps}\big)^{2}V'\big(\cW_{\eta}[\rho_{\eta}^{\eps}]\big)\partial_{y}\cW_{\eta}[\rho_{\eta}^{\eps}]\dd y 
\intertext{adding a zero}
&=\iint_{\OT}\phi(t,y)2\partial_{y}\rho_{\eta}^{\eps}V(\rho_{\eta}^{\eps})\rho_{\eta}^{\eps}\dd y+\iint_{\OT}\phi_{y}(t,y)\big(\rho_{\eta}^{\eps}\big)^{2}V(\rho_{\eta}^{\eps})\dd y\\
&\quad+\iint_{\OT}\phi_{y}(t,y)\big(\rho_{\eta}^{\eps}\big)^{2}\Big(V\big(\cW_{\eta}[\rho_{\eta}^{\eps}]\big)-V(\rho_{\eta}^{\eps})\Big)\dd y\\
&\quad+\iint_{\OT}\phi(t,y)\big(\rho_{\eta}^{\eps}\big)^{2}V'\big(\cW_{\eta}[\rho_{\eta}^{\eps}]\big)\partial_{y}\cW_{\eta}[\rho_{\eta}^{\eps}]\dd y 
\intertext{and an integration by parts in the first term}
&=-\iint_{\OT}\phi(t,y)\big(\rho_{\eta}^{\eps}\big)^{2}V'(\rho_{\eta}^{\eps})\partial_{y}\rho_{\eta}^{\eps}\dd y\\
&\quad+\iint_{\OT}\phi_{y}(t,y)\big(\rho_{\eta}^{\eps}\big)^{2}\Big(V\big(\cW_{\eta}[\rho_{\eta}^{\eps}]\big)-V(\rho_{\eta}^{\eps})\Big)\dd y\\
&\quad+\iint_{\OT}\phi(t,y)\big(\rho_{\eta}^{\eps}\big)^{2}V'\big(\cW_{\eta}[\rho_{\eta}^{\eps}]\big)\partial_{y}\cW_{\eta}[\rho_{\eta}^{\eps}]\dd y
\intertext{using the identity \(\rho_{\eta}^{\eps}\equiv \cW_{\eta}[q_{\eta}^{\eps}]-\eta\partial_{2}\cW_{\eta}[q_{\eta}^{\eps}]\) as derived in \cref{eq:identity_W} on the last summand}
&=-\iint_{\OT}\phi(t,y)\big(\rho_{\eta}^{\eps}\big)^{2}V'(\rho_{\eta}^{\eps})\partial_{y}\rho_{\eta}^{\eps}\dd y+\iint_{\OT}\phi_{y}(t,y)\big(\rho_{\eta}^{\eps}\big)^{2}\Big(V\big(\cW_{\eta}[\rho_{\eta}^{\eps}]\big)-V(\rho_{\eta}^{\eps})\Big)\dd y\\
&\quad+\iint_{\OT}\phi(t,y)\big(\cW_{\eta}[q_{\eta}^{\eps}]-\eta\partial_{y}\cW_{\eta}[q_{\eta}^{\eps}]\big)^{2}V'\big(\cW_{\eta}[\rho_{\eta}^{\eps}]\big)\partial_{y}\cW_{\eta}[\rho_{\eta}^{\eps}]\dd y\\
&=-\iint_{\OT}\phi(t,y)\big(\rho_{\eta}^{\eps}\big)^{2}V'(\rho_{\eta}^{\eps})\partial_{y}\rho_{\eta}^{\eps}\dd y+\iint_{\OT}\phi_{y}(t,y)\big(\rho_{\eta}^{\eps}\big)^{2}\Big(V\big(\cW_{\eta}[\rho_{\eta}^{\eps}]\big)-V(\rho_{\eta}^{\eps})\Big)\dd y\\
&\quad+\iint_{\OT}\phi(t,y)\cW_{\eta}[q_{\eta}^{\eps}]^{2}V'\big(\cW_{\eta}[\rho_{\eta}^{\eps}]\big)\partial_{y}\cW_{\eta}[\rho_{\eta}^{\eps}]\dd y\\
&\quad-2\eta\iint_{\OT}\phi(t,y)\cW_{\eta}[q_{\eta}^{\eps}]V'\big(\cW_{\eta}[\rho_{\eta}^{\eps}]\big)\partial_{y}\cW_{\eta}[\rho_{\eta}^{\eps}]^2\dd y\\
&\quad +\eta^{2}\iint_{\OT}\phi(t,y)\partial_{y}\cW_{\eta}[q_{\eta}^{\eps}]^{3}V'\big(\cW_{\eta}[\rho_{\eta}^{\eps}]\big)\dd y
\intertext{and resorting terms}
&=-\iint_{\OT}\phi(t,y)\big(\rho_{\eta}^{\eps}\big)^{2}V'(\rho_{\eta}^{\eps})\partial_{y}\rho_{\eta}^{\eps}\dd y+\iint_{\OT}\phi(t,y)\cW_{\eta}[\rho_{\eta}^{\eps}]^{2}V'\big(\cW_{\eta}[\rho_{\eta}^{\eps}]\big)\partial_{y}\cW_{\eta}[\rho_{\eta}^{\eps}]\dd y\\
&\quad+\iint_{\OT}\phi_{y}(t,y)\big(\rho_{\eta}^{\eps}\big)^{2}\Big(V\big(\cW_{\eta}[\rho_{\eta}^{\eps}]\big)-V(\rho_{\eta}^{\eps})\Big)\dd y\\
&\quad+\eta\iint_{\OT}\phi(t,y)V'\big(\cW_{\eta}[\rho_{\eta}^{\eps}]\big)\partial_{y}\cW_{\eta}[\rho_{\eta}^{\eps}]^2\Big(\eta\partial_{y}\cW_{\eta}[\rho_{\eta}^{\eps}]-2\cW_{\eta}[\rho_{\eta}^{\eps}]\Big)\dd y
\intertext{using \(-2\cW_{\eta}[\rho_{\eta}^{\eps}]+\eta\partial_{2}\cW_{\eta}[\rho_{\eta}^{\eps}]\equiv -\cW_{\eta}[\rho_{\eta}^{\eps}] -\rho_{\eta}^{\eps}\leqq 0\) according to \cref{eq:identity_W} and \(V'\leqq 0\) as well as \(\phi\geqq 0\)}
&\geq-\iint_{\OT}\phi(t,y)\big(\rho_{\eta}^{\eps}\big)^{2}V'(\rho_{\eta}^{\eps})\partial_{y}\rho_{\eta}^{\eps}\dd y+\iint_{\OT}\phi(t,y)\cW_{\eta}[\rho_{\eta}^{\eps}]^{2}V'\big(\cW_{\eta}[\rho_{\eta}^{\eps}]\big)\partial_{y}\cW_{\eta}[\rho_{\eta}^{\eps}]\dd y\\
&\quad+\iint_{\OT}\phi_{y}(t,y)\big(\rho_{\eta}^{\eps}\big)^{2}\Big(V\big(\cW_{\eta}[\rho_{\eta}^{\eps}]\big)-V(\rho_{\eta}^{\eps})\Big)\dd y.
 \end{align*}
The third term converges to zero for \(\eta\rightarrow 0\) following the estimate in \cref{eq:4212} and the first two terms can be written as
\begin{align}
    &\iint_{\OT}\phi(t,y)\cW_{\eta}[\rho_{\eta}^{\eps}]^{2}V'\big(\cW_{\eta}[\rho_{\eta}^{\eps}]\big)\partial_{y}\cW_{\eta}[\rho_{\eta}^{\eps}]\dd y-\iint_{\OT}\phi(t,y)\big(\rho_{\eta}^{\eps}\big)^{2}V'(\rho_{\eta}^{\eps})\partial_{y}\rho_{\eta}^{\eps}\dd y\notag\\
    &= \iint_{\OT}\phi_{y}(t,y)\big(G(\rho_{\eta}^{\eps})-G(\cW_{\eta}[\rho_{\eta}^{\eps}])\big)\dd y\label{eq:424242}
\end{align}
with 
\[
G:\begin{cases}\R&\rightarrow \R\\
            x&\mapsto \int_{0}^{x}s^{2}V'(s)\dd s.
            \end{cases}
\]
Obviously, \(G\in\sW^{1,\infty}_{\loc}(\R)\) and again -- in analogy to \cref{eq:4212} -- we can let \(\eps\rightarrow 0\) and later on \(\eta\rightarrow0\) to have
\begin{align*}
    |\eqref{eq:424242}|\leq \|\phi_{y}\|_{\sL^{\infty}((\OT))}\|G'\|_{\sL^{\infty}((0,\|vq_{0}\|_{\sL^{\infty}(\R)}))}\|\rho_{\eta}-\cW_{\eta}[\rho_{\eta}]\|_{\sC([0,T];\sL^{1}(\R))}\overset{\eta\rightarrow 0}{\longrightarrow} 0.
\end{align*}
This results in the fact that we obtain
\[
\lim_{\eta\rightarrow 0}\mEF\big[\phi,(\cdot)^{2},q_{\eta}\big]\geq 0
\]
and thus, the nonlocal solution satisfies in the limit \(\eta\rightarrow 0\) the Entropy condition.
 \end{proof}
\section{Convergence of the solution of the nonlocal discontinuous conservation law to the solution of the corresponding local conservation law}
 In this short section, we present the main result of this paper with all necessary conditions.
 \begin{theorem}[Convergence -- discontinuous nonlocal to local]\label{theo:convergence}
Let \(v\in\sL^{\infty}(\R;\R_{\geq \uv}),\ \uv\in\R_{>0}\) be given and assume in addition the OSL condition \cref{eq:OSL}, i.e.\ 
\[
\exists L\in\R\ \forall x,y\in\R,\ x\geq y:\ v(x)-v(y)\leq L(x-y) 
\]
Let furthermore \(q_{0}\in\sBV(\R;\R_{\geq0})\) be given and \(V\) may satisfy  \cref{eq:strict_concavity}, i.e., 
\[
x\mapsto x V(x) \text{ strictly concave }\Longleftrightarrow 2V'(x)+x V''(x)<0,\ \forall x\in \big[0, \|v\cdot q_{0}\|_{\sL^{\infty}(\R)}\big].
\]
Then, the solution of the discontinuous nonlocal conservation law \(q_{\eta}\in \sC\big([0,T];\sL^{1}(\R)\big)\) in \cref{defi:discontinuous_nonlocal_conservation_law} converges in \(\sC\big([0,T];\sL^{1}_{\loc}(\R)\big)\) to the Entropy solution of the (local) discontinuous conservation law \(q_{*}\in \sC\big([0,T];\sL^{1}_{\loc}(\R)\big)\) in \cref{defi:discontinous_local} if \(\eta\in\R_{>0}\) approaches zero, in formulae
\[
\lim_{\eta\rightarrow 0}\big\|q_{\eta}-q_{*}\big\|_{\sC([0,T];\sL^{1}_{\loc}(\R))}=0.
\]
Even more, also the nonlocal term converges in the following sense
\[
\lim_{\eta\rightarrow 0}\big\|\cW_{\eta}[v\cdot q_{\eta}]-v\cdot q_{*} \big\|_{\sC([0,T];\sL^{1}_{\loc}(\R))}=0.
\]
 \end{theorem}
 \begin{proof}
 For subsequences \((\eta_{k})_{k\in\N},\) this is a direct consequence of \cref{theo:convergence_entropy_solutions} and \cref{theo:strong_convergence_subsequences}. Thanks to the fact that the local entropy limit is unique, we can extend it straight forward to any sequence converging to zero.
 \end{proof}
 \begin{remark}[Convergence and generalizations]
Several points are in order to be made:
\begin{description}
    \item[\textbf{Convergence for each sequence}:] The convergence result in \cref{theo:convergence} holds for each sequence \(\eta_{k}\) converging to zero, and is thus a general approximation result.
    \item[\textbf{Convergence for more general kernels}:] As pointed out before, the methods used here are heavily inspired by \cite{coclite2022general}, but it is very likely that they can be extended to general kernels as done in \cite{colombo2022nonlocal} for the case without discontinuities. We do not go into details.
\end{description}
 \end{remark}

 \section{Numerical illustrations}\label{sec:numerics}
 In this section, we demonstrate numerically the claimed convergence in \cref{theo:convergence} and illustrate that the theorem might hold in more generality.
The numerical algorithm used is the one discussed in \cite{pflug2} which is non-dissipative, conserving the \(\sL^{1}\) mass by definition and which is based on the method of characteristics.

In the first example illustrated in \cref{fig:1} where we have a jump discontinuity located at \(x=0\) jumping downwards from \(\tfrac{3}{2}\) to \(\tfrac{1}{2}\) we can observe the suggested convergence and also the building of a rarefaction smoothing the introduced jump-discontinuity emanating from \(x=0\) over time.
 in the bottom row, the solution is scaled with the spatial discontinuity $v$ and one can observe the depicted \(v\cdot q\) satisfied the maximum principle as well as again the rarefaction forming for smaller \(\eta\) around \(x=0\).
In the two most left pictures, one can observe a kink at \(x=0\) around \(t\approx 0.35\) which comes from the fact that the mass has overstepped the discontinuous velocity at \(x=0\) so that it moves slower (actually \(\tfrac{1}{3}\) of the ``previous'' velocity).

In the second example illustrated in \cref{fig:2}, the jump, again located at \(x=0\), upwards from \(\tfrac{1}{2}\) to \(\tfrac{3}{2}\). And indeed, this causes the density right of \(x=0\) to move significantly faster then the left hand side. The blue triangle (where the density is \(\approx \tfrac{1}{6}\)) stems from the fact that density at \(x\leq 0\) enters this region and is scaled down by a factor of \(\tfrac{1}{3}\) as the density right of \(x=0\) moves with three times the speed in comparison to the density left of \(x=0\). Again, for smaller \(\eta\) the convergence can be observed (however, we could only show convergence only for discontinuous velocities which cannot jump upwards, see \cref{theo:strong_convergence_subsequences}).  The maximum principle for the quantity \(q\cdot v\) can be observed.
It can also be observed that in contrast to the examples in \cref{fig:1} no rarefaction at \(x=0\) builds up over time, which is in line with Oleinik's Entropy \cite{oleinik,oleinik_english} condition.

 \begin{figure}
 \centering
  
 \includegraphics[scale=0.4,trim=0 25 15 5,clip]{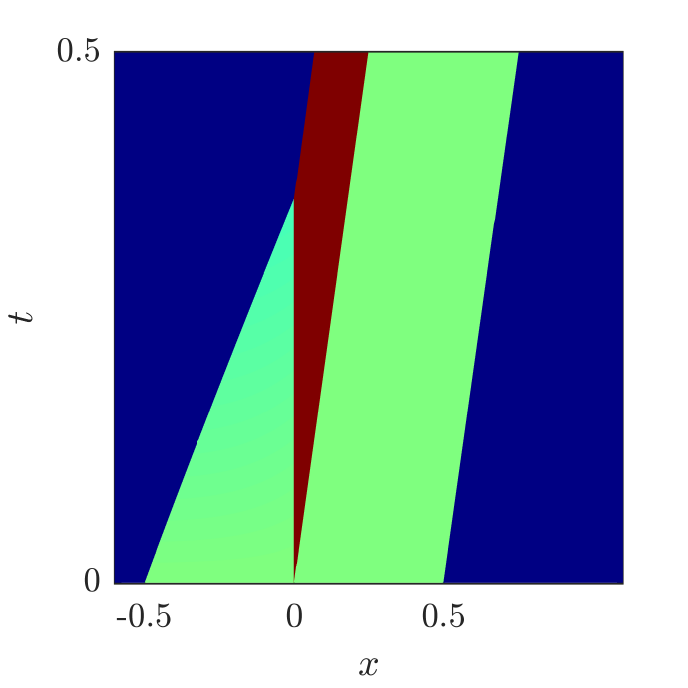}
  \includegraphics[scale=0.4,trim=30 25 15 5,clip]{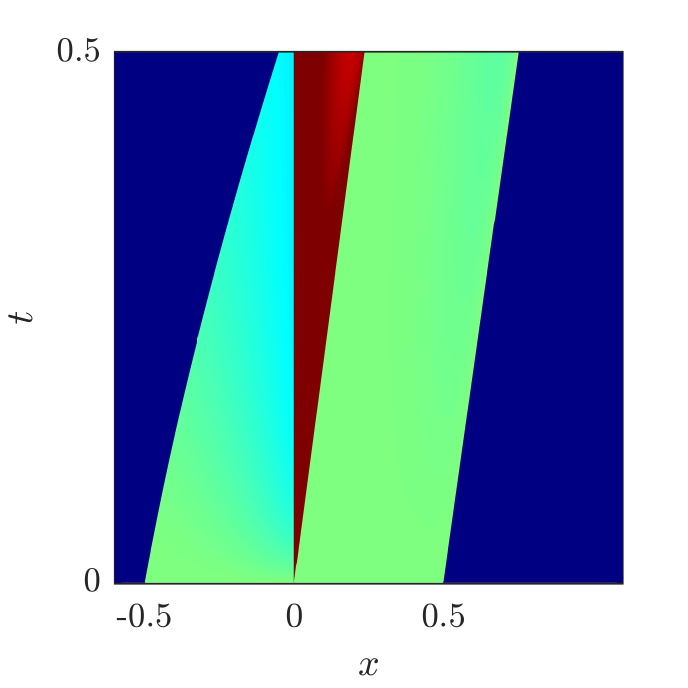}
   \includegraphics[scale=0.4,trim=30 25 15 5,clip]{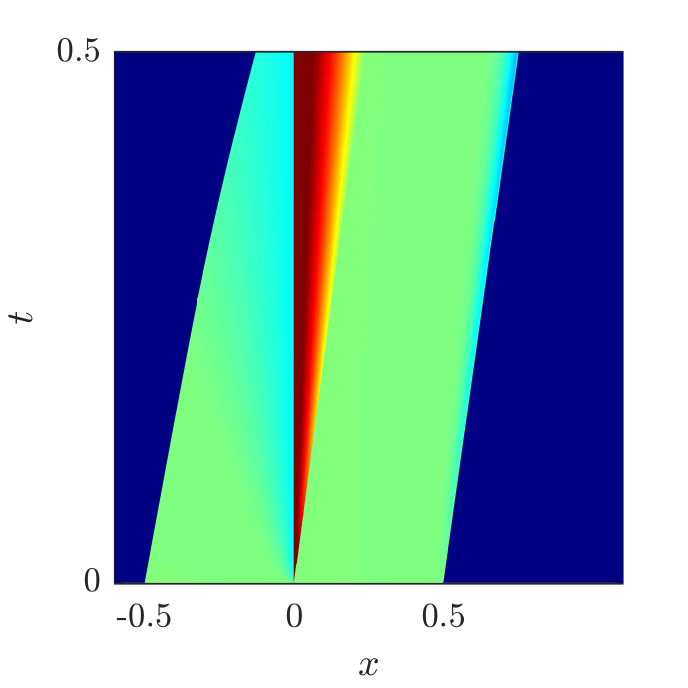}
     \includegraphics[scale=0.4,trim=30 25 15 5,clip]{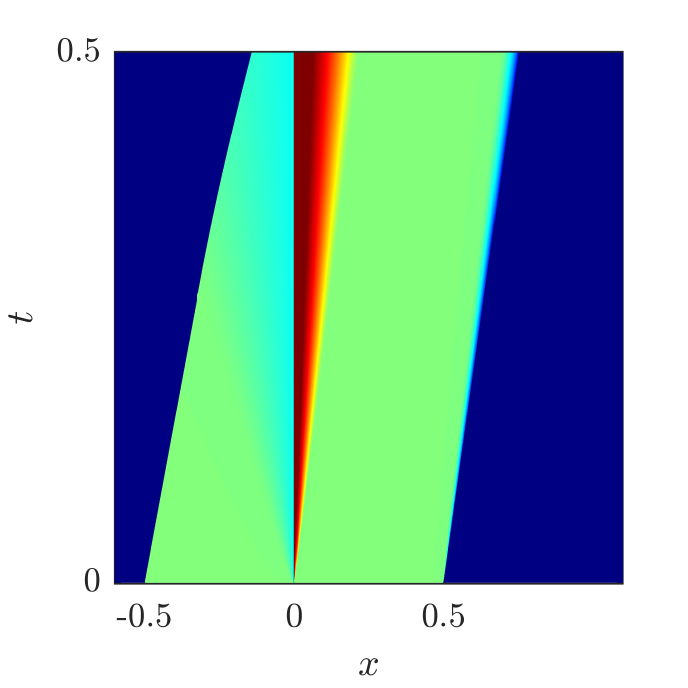}
        \includegraphics[scale=0.4,trim=30 25 10 0,clip]{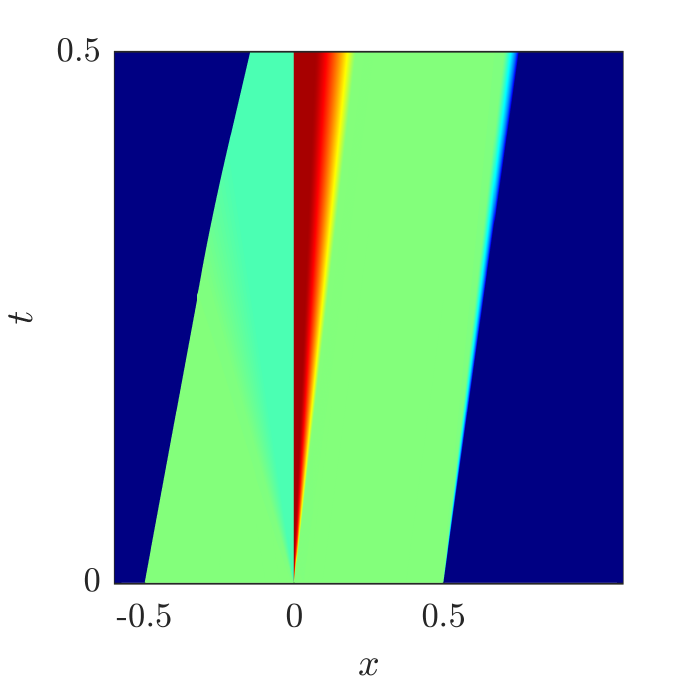}

 \includegraphics[scale=0.4,trim=0 0 15 5,clip]{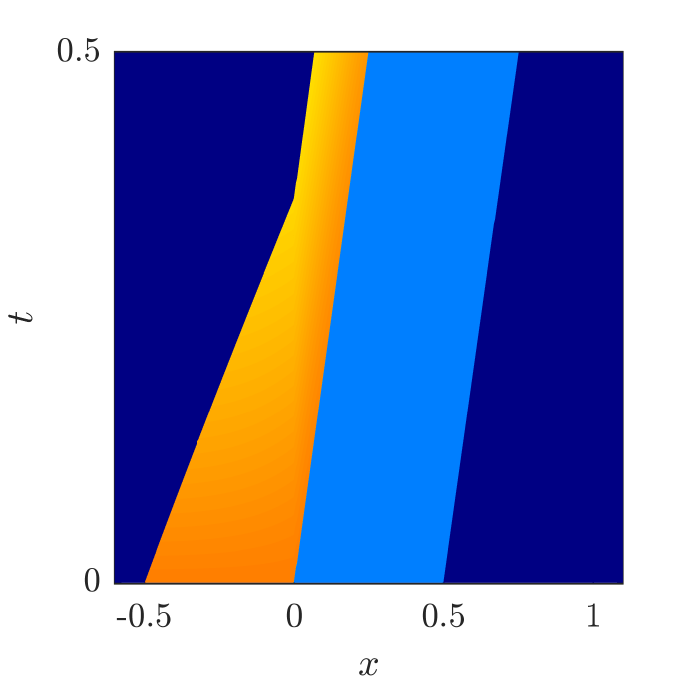}
  \includegraphics[scale=0.4,trim=30 0 15 5,clip]{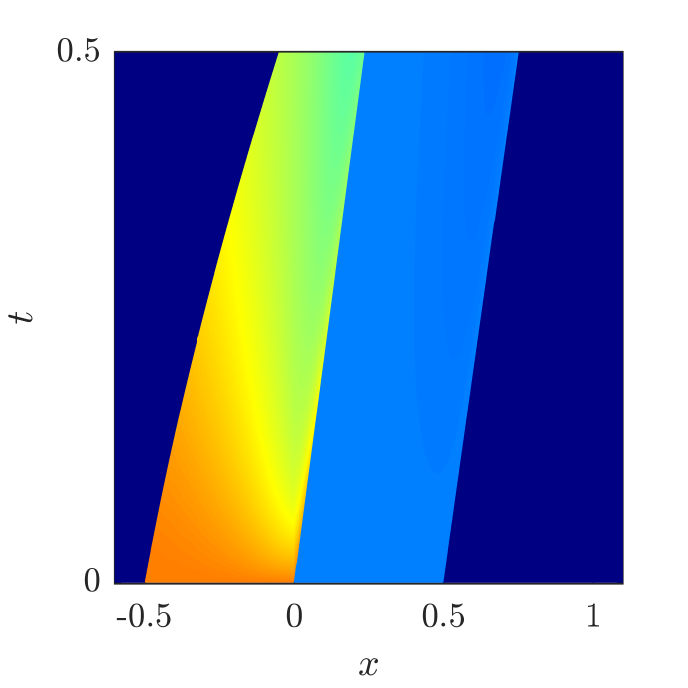}
\includegraphics[scale=0.4,trim=30 0 15 5,clip]{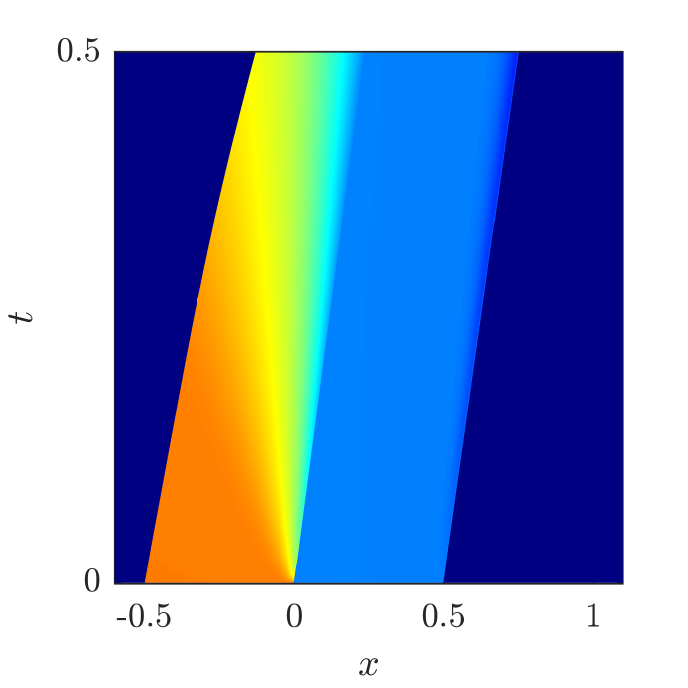}
 \includegraphics[scale=0.4,trim=30 0 15 5,clip]{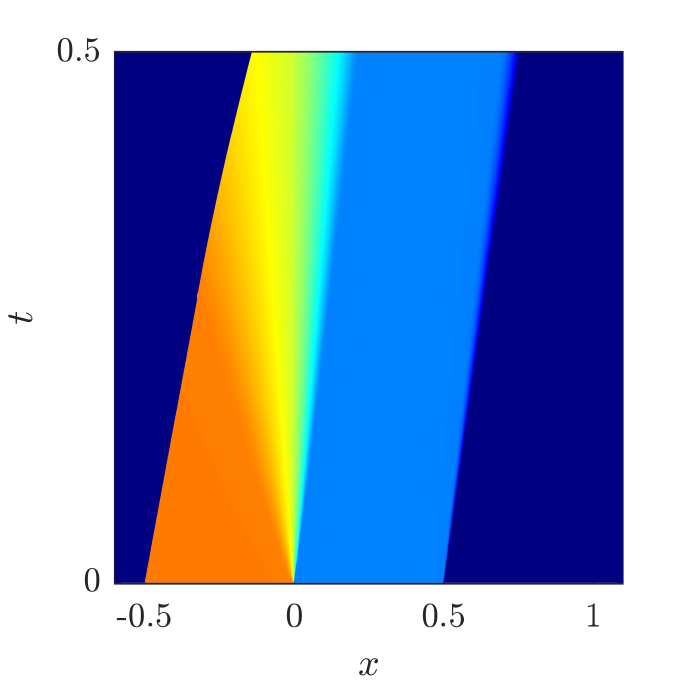}
\includegraphics[scale=0.4,trim=30 0 15 5,clip]{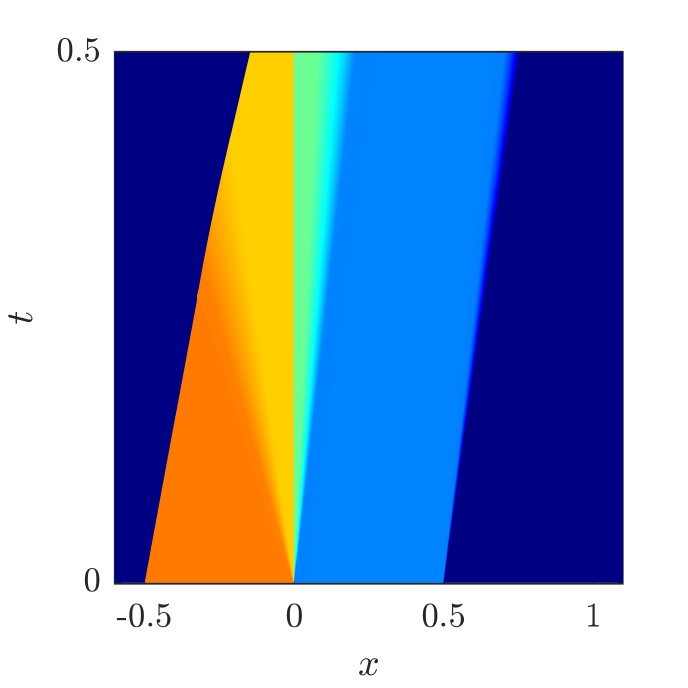}
        
    \caption{Velocity \(V(x)=1-x^{2}\), discontinuity \(v(x)=\tfrac{1}{2}+\chi_{\R_{<0}}(x),\ q_{0}(x)=\tfrac{1}{2}\chi_{[-0.5,0.5]}(x),\ x\in\R\),\ \(\eta\in \big\{10^{-i},i\in\{0,1,2,3,4\}\big\}\) from left to right. In the upper row, the solution \(q\) as in \cref{defi:discontinuous_nonlocal_conservation_law} is illustrated, while in the second row, \(\rho\equiv v\cdot q\) is shown (on which according to \cref{theo:existence_uniqueness_maximum} the maximum principle holds). Colorbar: $0\ $\protect\includegraphics[width=2cm]{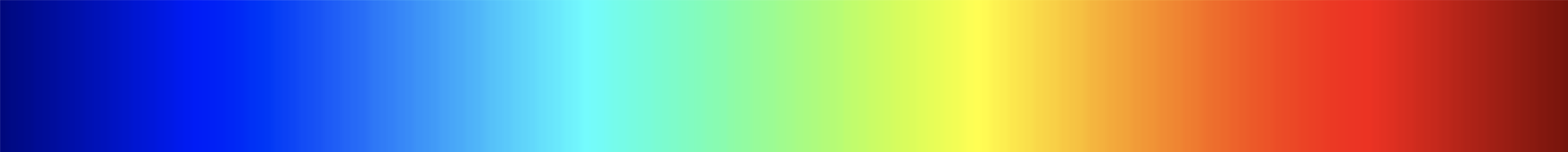}$\ 1$}
    \label{fig:1}
\end{figure}

\begin{figure}
 \centering
  
 \includegraphics[scale=0.4,trim=0 25 15 5,clip]{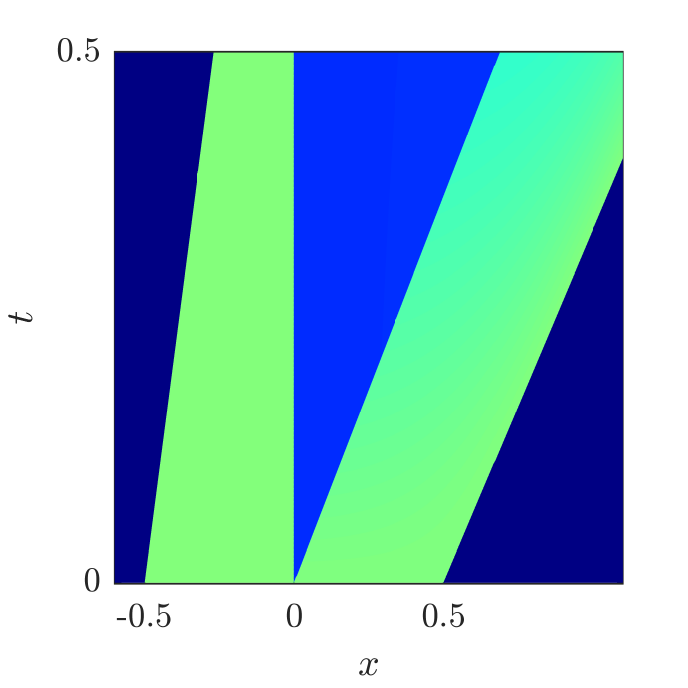}
  \includegraphics[scale=0.4,trim=30 25 15 5,clip]{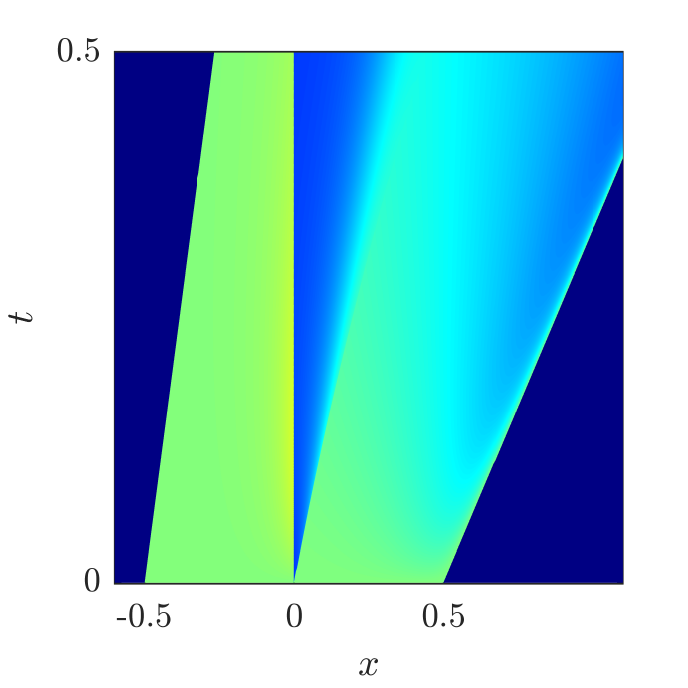}
   \includegraphics[scale=0.4,trim=30 25 15 5,clip]{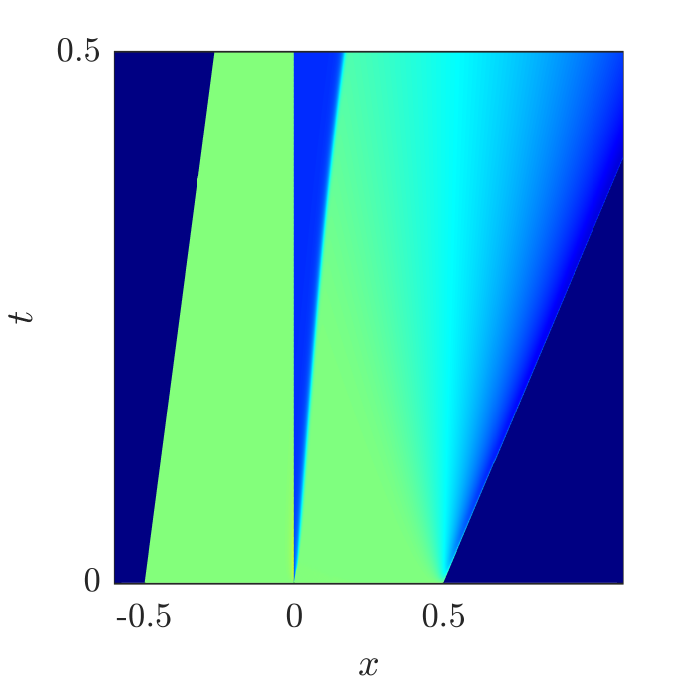}
     \includegraphics[scale=0.4,trim=30 25 15 5,clip]{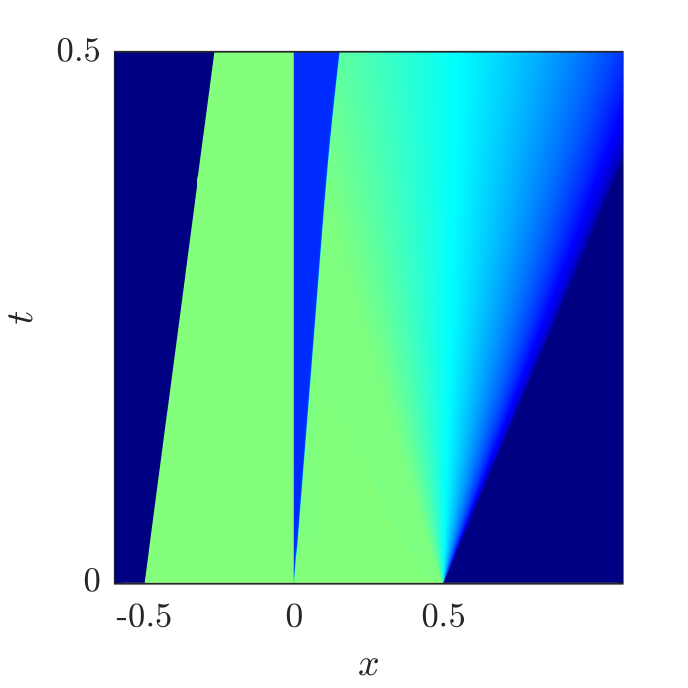}
        \includegraphics[scale=0.4,trim=30 25 10 0,clip]{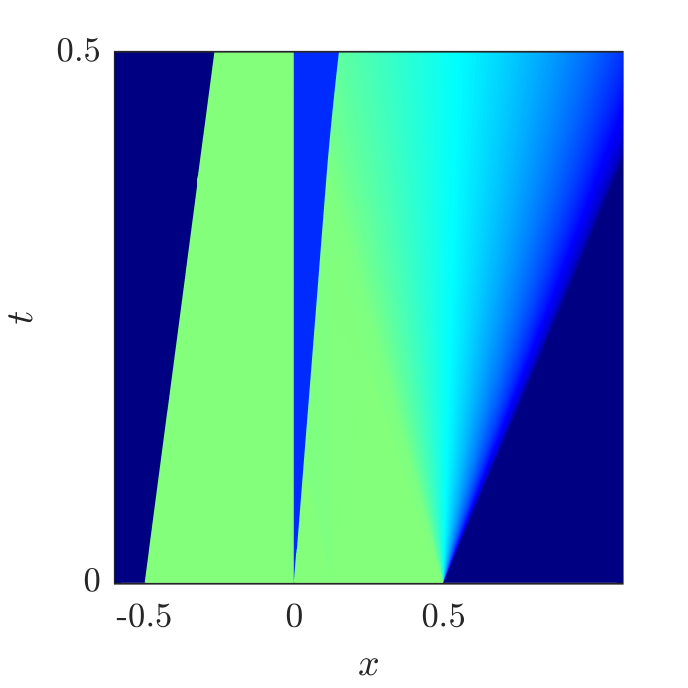}

 \includegraphics[scale=0.4,trim=0 0 15 5,clip]{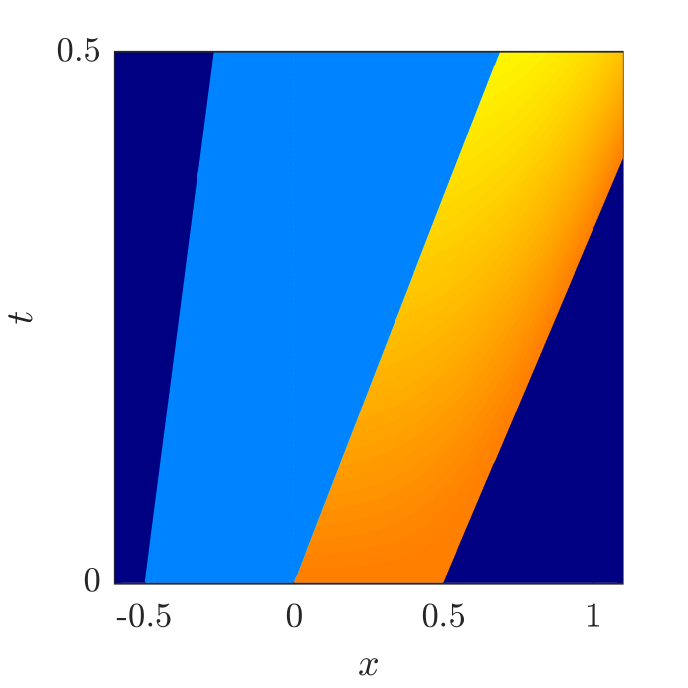}
  \includegraphics[scale=0.4,trim=30 0 15 5,clip]{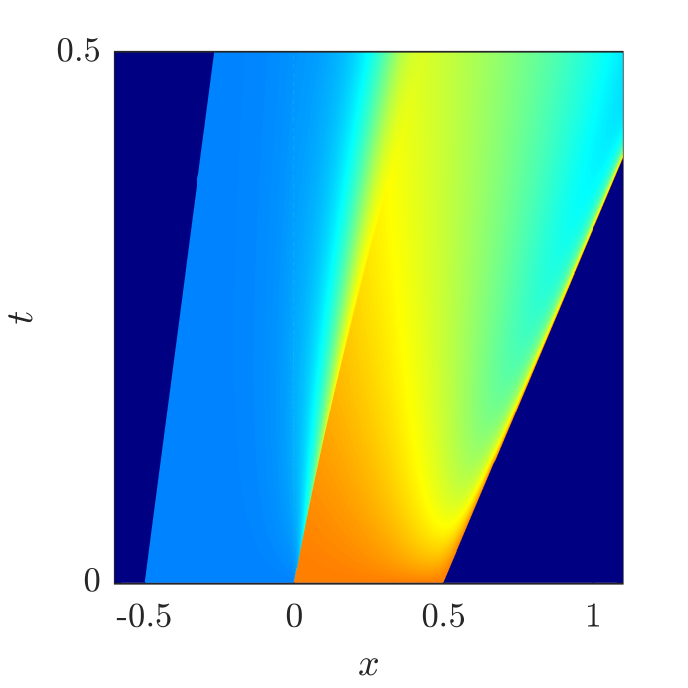}
\includegraphics[scale=0.4,trim=30 0 15 5,clip]{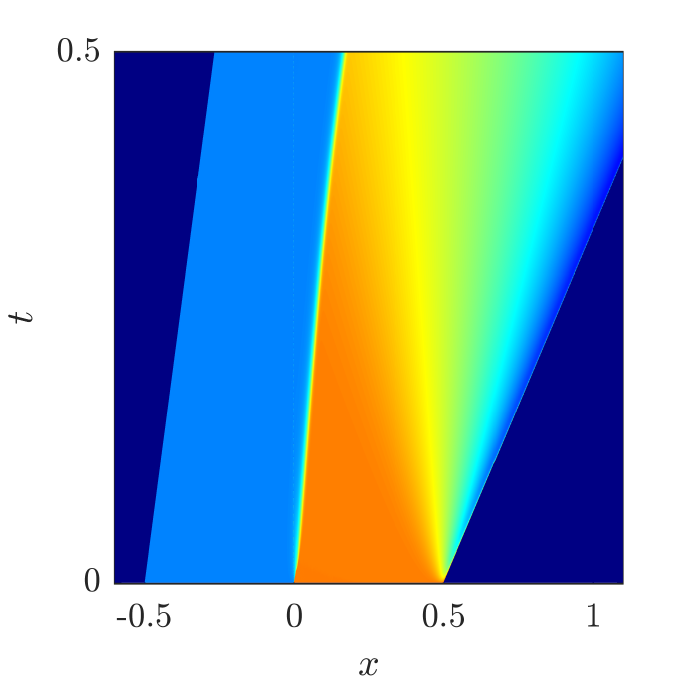}
 \includegraphics[scale=0.4,trim=30 0 15 5,clip]{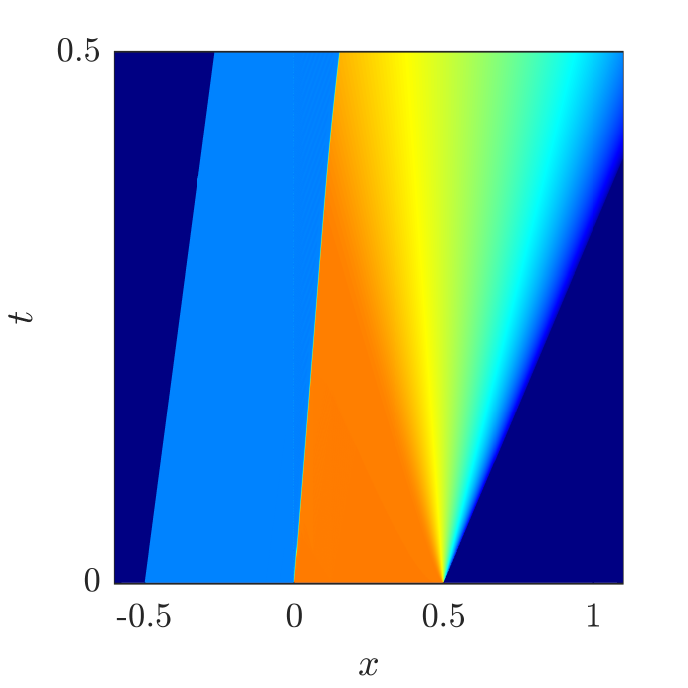}
\includegraphics[scale=0.4,trim=30 0 15 5,clip]{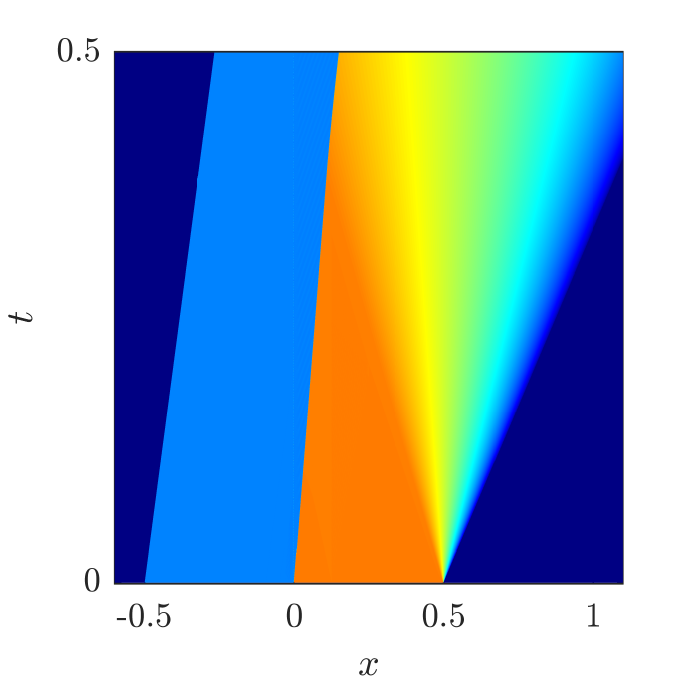}
\caption{Velocity \(V(x)=1-x^{2}\), discontinuity \(v(x)=\tfrac{1}{2}+\chi_{\R_{>0}}(x),\ q_{0}(x)=\tfrac{1}{2}\chi_{[-0.5,0.5]}(x),\ x\in\R\),\ \(\eta\in \big\{10^{-i},i\in\{0,1,2,3,4\}\big\}\) from left to right. In the upper row, the solution \(q\) as in \cref{defi:discontinuous_nonlocal_conservation_law} is illustrated, while in the second row, \(\rho\equiv v\cdot q\) is shown (on which according to \cref{theo:existence_uniqueness_maximum} the maximum principle holds). Colorbar: $0\ $\protect\includegraphics[width=2cm]{figures_discontinuous/ColorbarJet.png}$\ 1$}
\label{fig:2}
\end{figure}

\section{Conclusions and future work}
 We have established the convergence of solutions of a specific discontinuous nonlocal conservation law where the discontinuity is not only located in the velocity but also in the nonlocal term as a type of weight, and the nonlocal kernel approaches a Dirac distribution. On the local side, the emerging discontinuous local conservation law could be diffeomorphically transformed into a classical conservation law without discontinuity where typical Entropy conditions yield existence and uniqueness. However, this transform was not possible for the nonlocal discontinuous conservation law which justifies the outlined analysis furthermore.
 Typical \(\sTV\) estimates on the solution itself were not successful and proved in previous contributions as actual impossible \cite{COLOMBO20211653}, however, by relying on \(\sTV\) estimates of the nonlocal term we obtained strong convergence in \(\sC(\sL^{1}_{\loc})\).

 The application of this result is currently only driven of the mathematical interest to better understand the singular limit problem, but there might be a further interpretation as the discontinuous nonlocal conservation law might prove to be a rather good approximation for traffic flow with different road capacities/velocities (at the discontinuities). However, this requires a more application driven study on the suggested dynamics which is out of scope of this contribution.

\bibliographystyle{plain}
\bibliography{biblio}
\end{document}